\documentclass{siamltex}

\usepackage{epsfig,graphicx,psfrag,mathrsfs}
\usepackage{amssymb}
\usepackage{amsmath}
\usepackage[algoruled]{algorithm2e}
\usepackage{multirow}
\usepackage{color}

\usepackage{colortbl}
\usepackage{xcolor}

\usepackage{multirow}
\newtheorem{Remark}{Remark}[section]
\newtheorem{Lemma}{Lemma}[section]
\newtheorem{Algorithm}{Algorithm}[section]
\newtheorem{example}{{\it Example}}

\newcommand{\bb}{\begin{bmatrix}}
\newcommand{\eb}{\end{bmatrix}}

\title{
An accelerated technique for solving one type of discrete-time algebraic Riccati equations
}
\author{
Matthew M. Lin \thanks{ Department of Mathematics, National Cheng Kung University, Tainan 701, Taiwan. The first author was supported by the Ministry of Science and Technology of Taiwan under grants
104-2115-M-006-017-MY3 and 105-2634-E-002-001
(mhlin@mail.ncku.edu.tw). } \and Chun-Yueh Chiang
\thanks{Corresponding author.
Center for General Education, National Formosa
University, Huwei 632, Taiwan. The second author was supported by the Ministry of Science and Technology of Taiwan under grant 105-2115-M-150-001.  (chiang@nfu.edu.tw). }
}

\begin{document}
\maketitle

\begin{abstract}
Algebraic Riccati equations are encountered in many applications of control and engineering problems, e.g., LQG problems and $H^\infty$  control theory. In this work, we study the properties of one type of discrete-time algebraic Riccati equations. Our contribution is twofold. First, we present sufficient conditions for the existence of a unique positive definite solution. Second, we propose an accelerated algorithm to obtain the positive definite solution with the rate of convergence of any desired order. Numerical experiments strongly support that our approach performs extremely well even in the almost critical case.  As a byproduct, we provide show that this method is capable of computing the unique negative definite solution, once it exists.



\end{abstract}


%
%
%
\section{Introduction}

Originated from the study of control theory, the discrete-time algebraic Riccati equation (DARE) of the compact form:
\begin{align}\label{eq:DARE}
X=H+A^H {X} (I+ G{X})^{-1} A
\end{align}
has been extensively investigated; see~\cite{Pappas1980,Kimura1988,
Gudmundsson1992,Lu1993,Sun1998,Guo1999,Lu1999,Davies2008, WWL2006, Chiang2010,Rojas2013, Zhang2015,Miyajima2017} and the references therein.
In this work we would like to investigate the conjugate discrete-time algebraic Riccati equations (CDAREs)  in the form with the plus sign:
\begin{subequations}\label{eq:NME}
\begin{align}\label{eq:NMEP}
X=H+A^H \overline{X} (I+ G\overline{X})^{-1} A
\end{align}
and in the form with the minus sign:
\begin{align}\label{eq:NMEM}
X=H-A^H \overline{X} (I+ G\overline{X})^{-1} A,
\end{align}
\end{subequations}
generalized formulae for DAREs,
where $A \in \mathbb{C}^{n \times n}$, $G$ and $H$ are two Hermitian positive definite matrices with size $n\times n$, and the $n$-square matrix $X$ is an unknown Hermitian matrix and will be determined.
 %

In the paper, we derive some sufficient conditions for the existence of the unique positive solution.
Moreover, we present a numerical procedure, based on the fixed point iteration, to solve CDAREs, and show that the speed of convergence can be of any desired order.

An immediate question 
is whether this conjugate formulae~\eqref{eq:NME} could be equivalently transformed to the compact form~\eqref{eq:DARE}. %
To this end, we use the {notations}:
\begin{align}\label{eq:FX}
\mathcal{F}_{\pm}( X)={H}\pm{A}^H \overline{X} \Delta_{{G},\overline{X}} {A},
\end{align}
with $\Delta_{{G},{X}}:= (I+GX)^{-1}$ to simplify our discussion, that is,~\eqref{eq:NME} can also be represented by
\[
X = \mathcal{F}_{\pm}( X).
\]
Following from the fact that
\begin{equation}\label{eq:GF}
\Delta_{G,\overline{\mathcal{F}}_{\pm}( X)}=\Delta_{G,\overline{H}}\mp\Delta_{G,\overline{H}} G\overline{A}^H {X}\Delta_{{G_1},X} \overline{A} \Delta_{G,\overline{H}},
\end{equation}
it can be seen that
\begin{eqnarray}\label{eq:F2X}
\mathcal{F}_{\pm}^{(2)}( X)&:=& \mathcal{F}_{\pm}(\mathcal{F}_{\pm}( X)) = H\pm A^H\overline{\mathcal{F}}_{\pm}( X) \Delta_{G,\overline{\mathcal{F}}_{\pm}( X)}A\nonumber\\
&=& H_1\pm(\Pi_1+\Pi_2+\Pi_3),
\end{eqnarray}
where
\begin{eqnarray}\label{HG}
\Pi_1&=&\pm{A}^H\overline{A}^H {X} \Delta_{\overline{G},{X}} \overline{A}\Delta_{G,\overline{H}}{A},\nonumber\\
\Pi_2&=&\mp{A}^H\overline{H}\Delta_{G,\overline{H}} G\overline{A}^H {X}\Delta_{{G_1},X} \overline{A} \Delta_{G,\overline{H}}{A},\nonumber\\
\Pi_3&=&-{A}^H \overline{A}^H {X} \Delta_{\overline{G},{X}} \overline{A}\Delta_{G,\overline{H}} G\overline{A}^H {X}\Delta_{{G_1},X} \overline{A} \Delta_{G,\overline{H}}{A},\nonumber\\
G_1&=&\overline{G}\pm\overline{A}\Delta_{{G},\overline{H}}G\overline{A}^H,\label{G1}\\
H_1&=&{H}\pm{A}^H \overline{H}\Delta_{{G},\overline{H}}{A}\label{H1}.
\end{eqnarray}
Note that~\eqref{eq:GF} is an application of  the well-known Sherman Morrison Woodbury formula, which can be stated as follows.
%
%
%
%
%
%
\begin{Lemma}\cite{Bernstein2009}\label{Schur}
Let $A$ and $B$ be two arbitrary matrices of size $n$, and  let $X$ and $Y$ be two $n\times n$ nonsingular matrices. Assume that $Y^{-1}\pm BX^{-1}A$ is nonsingular. Then, $X\pm A Y B$ is invertible and
 \[
(X\pm A Y B)^{-1}=X^{-1}\mp X^{-1}A(Y^{-1}\pm B X^{-1}A)^{-1}BX^{-1}.
\]
\end{Lemma}
%
%
%
%
%
%
%
%
%
%
We further observe that
\begin{align*}
\Pi_1+\Pi_3&=\pm{A}^H \overline{A}^H {X} \Delta_{\overline{G},{X}}\left(I_n\mp \overline{A}\Delta_{G,\overline{H}} G\overline{A}^H {X}\Delta_{{G_1},X}\right) \overline{A} \Delta_{G,\overline{H}}{A}\\
&=\pm{A}^H \overline{A}^H {X} \Delta_{\overline{G},{X}}\Delta_{\overline{G},{X}}^{-1}\Delta_{{G_1},X}\overline{A} \Delta_{G,\overline{H}}{A}\\
&=\pm{A}^H \overline{A}^H {X}\Delta_{{G_1},X}\overline{A} \Delta_{G,\overline{H}}{A}.
\end{align*}
Thus, we have
\begin{align*}
\Pi_1+\Pi_2+\Pi_3&= {A}^H\left(\pm I_n\mp\overline{H}\Delta_{G,\overline{H}} G\right) \overline{A}^H{X}\Delta_{{G_1},X}\overline{A} \Delta_{G,\overline{H}}{A}\\
&=\pm A_1^H X \Delta_{{G_1},X} A_1,
\end{align*}
where
\begin{equation}\label{A1}
A_1:=\overline{A}\Delta_{G,\overline{H}}{A}.
\end{equation}
This concludes that~\eqref{eq:NME} can be transformed into the standard DAREs
\begin{equation}\label{eq:DAREs}
X =H_1+A_1^H X\Delta_{G_1, X}A_1 .
\end{equation}

Starting with a fixed point iteration, we propose a 3-term iterative method in Section~\ref{sec:IM}. We show that this method has a semigroup property and is equivalent to the structured doubling algorithm (SDA), i.e.,
\begin{eqnarray*}
A_{k+1}&=&A_{k}(I+G_{k}H_{k})^{-1}A_{k},\\
G_{k+1}&=&G_{k}+A_{k}(I+G_{k}H_{k})^{-1} G_kA_{k}^H,\\
H_{k+1}&=&H_{k}+A_{k}^H H_{k}(I+G_{k}H_{k})^{-1}A_{k},
\end{eqnarray*}
under a specific transformation. Though the SDA is known for its efficiency of computing
the solution of DARE~\cite{WWL2006} with quadratic convergence, we use this semigroup property to build up an accelerated iterative method with the rate of convergence of any desired order.

The paper is organized as follows.
In Section~\ref{sec:SP} and Section~\ref{sec:IM}, we propose, respectively, sufficient conditions for the existence of unique positive definite solutions of~\eqref{eq:NME} by means of the solvable analysis of~\eqref{eq:DARE}.  Based on the fixed point iteration, we construct a way to solve the unique positive definite solutions of~\eqref{eq:NME}.  We show in Theorem~\ref{trans} that this way
satisfies a semigroup property. In Section~\ref{Sec:acc}, we apply this property to build up an accelerated approach to compute the positive definite solution with {r-superlinear convergence of order $r$, for any integer $r > 1$}. In Section~\ref{sec:NE}, we examine two examples to illustrate the capacity and efficiency of our proposed accelerated technique. In Section~\ref{sec:con}, we make our concluding remarks.

%
%
{In the subsequent discussion, the symbols $\mathbb{C}^{n \times n}$ and $\mathbb{P}_{n}$ stand for the set of $n\times n$ complex matrices and positive definite matrices, respectively.} We denote the $m\times m$ identity matrix by $I_m$, the conjugate matrix of $A$ by $\overline{A}$, the conjugate transpose matrix of $A$ by $A^H$, the spectrum of $A$ by $\sigma(A)$ and use $\rho(A)$ to denote the spectral radius of a square matrix $A$.
 We use the symbol $A> 0$ (or $A\geq 0$) to represent that $A$ is a Hermitian positive definite matrix (or a Hermitian positive semidefinite matrix) and {the Loewner order} $A > B$ (or $A\geq B$ ) if $A - B > 0$ (or $A-B\geq 0$). A matrix operator $f$ is order preserving on $\mathbb{P}_{n}$ if $f(A)\geq f(B)$ when $A\geq B$ and $A,B\in\mathbb{P}_{n}$.

\section{Solvability properties}\label{sec:SP}
In this section, we present sufficient conditions for unique existence of the positive definite solutions of~\eqref{eq:NME}. To this end, we start by investigating the solvability of the standard conjugate Stein matrix equation:
\begin{align}\label{CS}
X=Q+A^H\overline{X}A,
\end{align}
where $A\in\mathbb{C}^{n\times n}$ and $Q\in\mathbb{P}_n$.

Its proof is based on the following well-known fact.
%
\begin{Lemma}\label{MCT}~\cite[Proposition 8.6,3.]{Bernstein2009}
Let $\{A_i\}_{i=1}^\infty$ be a sequence of positive semidefinite matrices satisfying $A_j\geq A_i \geq 0$ if  $j\geq i$, and assume that $B$ is another positive semidefinite matrix satisfying
$B \geq A_i$ for all $i > 0$. Then,
$A:=\lim\limits_{i\rightarrow \infty} A_i$ exists and
$B\geq A \geq 0$.
\end{Lemma}

Upon using Lemma~\ref{MCT}, our next result is to propose a necessary and sufficient condition for the existence of a unique positive definite solution of~\eqref{CS}. 

\begin{Lemma}\label{conj}
The equation~\eqref{CS} has a unique positive definite solution if and only if
$\rho (\overline{A}A ) < 1$.
\end{Lemma}
\begin{proof}
Assume that $X_p$ is the unique positive definite solution of~\eqref{CS}. Thus,
$X_p$ is a solution of the equation:
\begin{align}\label{SCS}
X=Q+A^H\overline{Q}A+(\overline{A}A)^H{X}(\overline{A}A).
\end{align}
This implies that for any integer $k>0$,
\begin{equation}\label{SCP2}
X_p = \sum\limits_{i=0}^k\left((\overline{A}A)^i\right)^H(Q+A^H\overline{Q}A)(\overline{A}A)^i
+ \left((\overline{A}A)^{k+1}\right)^H  X_p (\overline{A}A)^{k+1} >0.
\end{equation}
Since $Q+A^H\overline{Q}A>0$ and $X_p$ is positive definite, we see that
\[
%
\sum\limits_{i=0}^\infty\left((\overline{A}A)^i\right)^H(Q+A^H\overline{Q}A)(\overline{A}A)^i
\]
converges, and hence  $\rho( \overline{A}A)<1$.

Conversely, assume that $\rho(\overline{A}A) < 1$. {Let $A\otimes B$ be the Kronecker product of matrices $A$ an $B$.
}
 Observe
from~\eqref{SCS}
that
\[
\left(I - (\overline{A}A)^\top \otimes (\overline{A}A)^H \right)\rm{vec}(X) = \rm{vec}(Q+A^H\overline{Q}A),
\]
where $\rm{vec}(\cdot)$  is the column stretching function defined as
\[
\rm{vec}(A)
= [a_{11},\cdots,a_{m1},\cdots,a_{1n},\cdots,a_{mn}]^\top
\]
 for any $m\times n$  matrix $A = [a_{ij}]$.
 This implies that the solution, say $X_p$, of~\eqref{SCS} exists. Also, from~\cite[Lemma 12]{Zhou2011}, we know that
 ~\eqref{CS} has a solution if $\rho(\overline{A}A) < 1$.
 Following from~\eqref{SCP2}, we have
 \begin{equation*}
 X_p =
 \sum\limits_{i=0}^\infty\left((\overline{A}A)^i\right)^H(Q+A^H\overline{Q}A)(\overline{A}A)^i,
 \end{equation*}
 which is positive definite. Once~\eqref{SCS} has a unique positive definite solution, this solution is also the unique positive definite solution of~\eqref{CS}. This completes the proof.

%
%

\end{proof}

Note that Lemma~\ref{conj} enables us to discuss the solvability of~\eqref{eq:NME}. To make our discussion more clearly and explicitly, the rest of this section is divided into two parts, respectively: One is for~\eqref{eq:NMEP} and the other is for~\eqref{eq:NMEM}.

%

\subsection{The solvability of~\eqref{eq:NMEP}}\label{2.1}
Using the formula in~\eqref{eq:FX}, let
$P_1$ be a set defined by
\begin{equation}\label{SetP1}
P_1
=\left \{X >0 | X \geq\mathcal{F}_{+}(X)\right\}.
\end{equation}
Consider the fixed point iteration
\begin{equation}\label{eq:F1}
X_{k+1}=\mathcal{F}_{+}(X_k)
\end{equation}
with $X_1=H$. It is easy to see that $\{X_k\}$ is a monotone increasing matrix sequence with respect to the Loewner order.
Once $P_1$ is nonempty, choose a matrix $X_{P_1}$ in $P_1$. It can be shown by induction that for any integer $k>0$, $X_k\leq X_{P_1}$.
This is because for $k=1$, it is true that $X_{P_1} \geq H=X_1$. Assume that this statement is true for $k= n$. Then,
\begin{eqnarray*}
 X_{P_1} &\geq& H+A^H \overline{X}_{P_1}\Delta_{G,\overline{X}_{P_1}} A \\
 &\geq&
 X_{n+1} +
 A^H\left ( (I+\overline{X}_{P_1}  G)^{-1} \overline{X}_{P_1}  - (I+\overline{X}_n G)^{-1} \overline{X}_n \right)A\\
 &=&X_{n+1}+A^H\left( (\overline{X}_{P_1}^{-1}+G)^{-1}-(\overline{X}_n^{-1}+G)^{-1} \right)A \geq X_{n+1}.
\end{eqnarray*}
Hence, the sequence $\{X_k\}$ converges, i.e.,
\begin{equation}\label{eq:Fsol}
X_\ast:= \lim\limits_{k\rightarrow \infty}X_k
\end{equation}
exists and satisfies~\eqref{eq:NMEP}.

In addition, let
\begin{equation}\label{Tx}
T_X=\Delta_{G,\overline{X}}A,
\end{equation}
and
\begin{equation}\label{wTx}
\widehat{T}_X = \overline{T}_X T_X.
\end{equation}
It can be seen that for this $X_{P_1}\in P_1$, we know that
\begin{align*}
X_{P_1}-T_{X_{P_1}}^H \overline{X}_{P_1} T_{X_{P_1}}
&=X_{P_1}-
A^H (I-\overline{X}_{P_1} G(I+\overline{X}_{P_1} G)^{-1})\overline{X}_{P_1}\Delta_{G,\overline{X}_{P_1}}A\\
&=X_{P_1}-A^H\overline{X}_{P_1}T_{X_{P_1}}+A^H\overline{X}_{P_1} G\Delta_{\overline{X}_{P_1},G} \overline{X}_{P_1}\Delta_{G,\overline{X}_{P_1}}A\\
 &\geq H+(\Delta_{\overline{X}_{P_1},G}
 \overline{X}_{P_1}A)^H G
(\Delta_{\overline{X}_{P_1},G}
 \overline{X}_{P_1}A),
\end{align*}
which yields
\[
X_{P_1}\geq \widehat{T}_{X_{P_1}}^H X_{P_1}\widehat{T}_{X_{P_1}}+H,
\]
or, equivalently,
\[
X_{P_1}\geq \sum\limits_{k=0}^m(\widehat{T}_{X_{P_1}}^H)^k H \widehat{T}_{X_{P_1}}^k
\]
for any integer $m>0$. This implies that the specific matrix computation $\widehat{T}_{X_{P_1}}$ satisfying
\[
\rho(\widehat{T}_{X_{P_1}})<1.
\]
In particular, it can be seen that if $X$ solves~\eqref{eq:NMEP},
\begin{align*}
\widehat{T}_X&
=\Delta_{\overline{G},{X}} \overline{A}\Delta_{G,\overline{X}}A\\
& {= \Delta_{\overline{G},{X}}
\overline{A}
\left(I+G
\left(\overline{H} +
\overline{A}^H X \Delta_{\overline{G},X}\overline{A} \right) \right)^{-1}
A
}\\
%
& {= \Delta_{\overline{G},{X}}
\overline{A}
\left(\Delta_{G,\overline{H}}
- \Delta_{G,\overline{H}} G\overline{A}^H
X
\left( I+(\overline{G} +\overline{A}
\Delta_{G,\overline{H}}G\overline{A}^H)X
\right)^{-1}
\overline{A}
\Delta_{G,\overline{H}}
\right)A
}\\
& {= \Delta_{\overline{G},{X}}
\overline{A}
(\Delta_{G,\overline{H}}
- \Delta_{G,\overline{H}} G\overline{A}^H
X
\Delta_{G_1, X}
\overline{A}
\Delta_{G,\overline{H}}
)A
}\\
&=\Delta_{\overline{G},{X}}  A_1-\Delta_{\overline{G},{X}}  \overline{A}\left(\Delta_{{G},\overline{H}}G \overline{A}^H X\Delta_{G_1,X} \overline{A}\Delta_{{G},\overline{H}} \right) A \\
&=\Delta_{\overline{G},{X}} \left(I+G_1 X- \overline{A}\Delta_{{G},\overline{H}}G \overline{A}^H X \right)\Delta_{G_1,X} A_1 \\
&=\Delta_{G_1,X} A_1.
\end{align*}

To make it clearly, we summarize results as follows.
\begin{theorem}\label{1ALremark}
Let $P_1$, ${T}_X$, and $\widehat{T}_X$ be the notation defined in~\eqref{SetP1},~\eqref{Tx},
and~\eqref{wTx}, respectively.
\begin{itemize}
\item[\rm{(a)}] If $P_1$ is nonempty, then there exists a positive definite solution of~\eqref{eq:NMEP}.
\item [\rm{(b)}] If $X\in P_1$, then $\rho(\widehat{T}_X)<1$.

\item [\rm{(c)}] If $X$ solves Eq.~\eqref{eq:NMEP}, then $\widehat{T}_X= 
\Delta_{G_1,X}A_1$.
\end{itemize}

\end{theorem}


%
%
%

Inspired by our above findings, we now propose a necessary and sufficient condition for the existence and uniqueness of the positive definite solution of~\eqref{eq:NMEP}.
\begin{theorem}\label{thm:pd2a}
The set $P_1$ is nonempty if and only if there exists a unique positive definite solution of~\eqref{eq:NMEP}.
\end{theorem}
\begin{proof}
If $P_1$ is nonempty, Theorem~\eqref{1ALremark} implies that there exists a positive definite solution of~\eqref{eq:NMEP}.
Next, we show that the positive definite solution of~\eqref{eq:NMEP} is unique.
To this end, let $X_1$ and $X_2$ be two positive definite solutions of~\eqref{eq:NMEP}. It follows that
\begin{eqnarray*}
X_1-X_2
&=& {A^H (I+\overline{X}_1 G)^{-1}
(
\overline{X}_1
(I+G\overline{X}_2 ) - (I+\overline{X}_1 G)
 \overline{X}_2)(I+G\overline{X}_2)^{-1} A}\\
&=&
T_{X_1}^H (\overline{X}_1-\overline{X}_2) T_{X_2}.
\end{eqnarray*}
Subsequently, we have
\[
X_1-X_2=(\widehat{T}_{X_1}^H)^k ({X_1}-{X_2}) \widehat{T}_{X_2}^k
\] for any integer $k>0$, which gives rise to the fact that
 \[
 X_1-X_2=\lim\limits_{k\rightarrow\infty}(\widehat{T}_{X_1}^H)^k ({X_1}-{X_2}) \widehat{T}_{X_2}^k=0.
 \]
This is because $X_1$ and $X_2$ are in $P_1$ and from Theorem~\ref{1ALremark} (b), we know that $\rho(\widehat{T}_{X_1})<1$ and $\rho(\widehat{T}_{X_2})<1$.

Conversely, if there exists a unique solution of~\eqref{eq:NMEP}, it is trivial that $P_1$ is nonempty.
\end{proof}

Note that Theorem~\ref{thm:pd2a} provides a necessary and sufficient condition for the existence of a unique positive definite solution of~\eqref{eq:NMEP}. However, the assumption $P_1\neq\phi$ is not easy to check. A useful sufficient condition for the existence of a unique positive definite solution of~\eqref{eq:NMEP} can be written as follows.

%

%
%
\begin{corollary}\label{thm:pd1a}
Assume that the coefficient matrix $A$ in~\eqref{eq:NMEP} satisfies
$\rho(\overline{A}A)<1$. Then, there exists a unique positive definite solution to~\eqref{eq:NMEP}.
\end{corollary}
\begin{proof}
Since $\rho(\overline{A}A)<1$, it follows from
Lemma~\ref{conj} that there exists a positive definite matrix $X_1$ such that
\[
X_1=H+A^H \overline{X}_1 A \geq H+A^H \overline{X}_1 A-(\overline{X}_1 A)^H  (G^{-1}+\overline{X}_1)^{-1} (\overline{X}_1 A)=\mathcal{F}_{+}(X_1).
\]
Thus, $P_1$ is nonempty. From Theorem~\ref{1ALremark}, there exists a positive definite solution of~\eqref{eq:NMEP}.

\end{proof}


%

{

\subsection{The solvability of~\eqref{eq:NMEM}}\label{2.2}

In this section, we discuss a counterpart of~\eqref{eq:NMEP}. To start with, we let $P_2$ be a set defined by
{
\begin{equation}\label{SetP2}
P_2:=\left\{ X > 0 |H\geq  X \geq\mathcal{F}_{-}(X)
\right\},
\end{equation}
}
and let $H_1$,  $G_1$, and $A_1$ be matrices defined in~\eqref{H1},~\eqref{G1} and~\eqref{A1} with minus signs.
Note that the set $P_2$ is nonempty, since $H\in P_2$.

Our purpose in this section is to show that there exists one and only one positive matrix $X\in P_2$, and $X$ satisfies~\eqref{eq:NMEM} and $\rho(\widehat{T}_X)<1$.
To prove these facts and make this work self-contained, we recall the result for nonlinear matrix equations in~\cite[Lemma~5.5]{Sayed01} and~\cite[Theorem~5.6]{Sayed01}.

\begin{theorem}\label{SayedRM}
Let $\mathcal{F}(X) = -X+X\mathcal{H}(X) X$ be an order preserving mapping of $\mathbb{P}_n$ into $n\times n$ negative definite matrices.
 Assume that  $\mathcal{H}$ satisfies the following two properties:
\begin{eqnarray*}
\mathcal{H}(X) X \mathcal{H}(X) & \leq& \mathcal{H}(X),\\
\mathcal{H}(Y)-\mathcal{H}(X) &=&%
\mathcal{H}(X)(X-Y)\mathcal{H}(Y).
\end{eqnarray*}
Then,
there is a unique positive definite solution $X$ to the equation
\[
X - A^H X A + A^H X\mathcal{H}(X)XA = H,
\]
where $A,H\in\mathbb{C}^{n\times n}$ and $H\geq 0$. Moreover, for this solution $X$, the spectrum radius of the matrix $\widehat{{T}}_{X}$ defined by
\[
\widehat{{T}}_{X} = A - \mathcal{H}({X}){X}A
\]
satisfying $\rho(\widehat{{T}}_{X}) < 1$.


\end{theorem}

Corresponding to~\eqref{eq:NME}, we consider the case that $\mathcal{F}(X) = -X + X \mathcal{H}(X) X$, where $\mathcal{H}(X) = \Delta_{G_1,X} G_1$ and show that this $\mathcal{F}(X)$ satisfies the requirement of Theorem~\ref{SayedRM}.

\begin{corollary}\label{cor:main}
Let $\mathcal{F}(X) = -X + X \mathcal{H}(X) X$ be a
mapping of $\mathbb{P}_n$,
where $\mathcal{H}(X) = \Delta_{G_1,X} G_1$ and $G_1>0$.
Then,
\begin{itemize}
\item [\rm{(a)}] $\mathcal{F}(X) = -X\Delta_{G_1,X}$, i.e.,
$0\leq X \leq Y$ implies that
$\mathcal{F}(X) \geq \mathcal{F}(Y)$.

\item [\rm{(b)}] $\mathcal{H}(X) X \mathcal{H}(X)  \leq \mathcal{H}(X)$ and
$\mathcal{H}(Y)-\mathcal{H}(X) =%
\mathcal{H}(X)(X-Y)\mathcal{H}(Y)$.

\item [\rm{(c)}] There is a unique positive definite solution $X$ to the DARE
\begin{equation*}\label{eq:h1}
X - A_1^H X\Delta_{G_1,X}A_1 =H_1,
\end{equation*}
where $H_1>0$.
Moreover, for this solution $X$,
$\rho(\widehat{{T}}_{X})  <1$ with the matrix $\widehat{{T}}_{X}$ defined by
$\widehat{{T}}_{X} = \Delta_{G_1,{X}}A_1$.

%
%

\end{itemize}

\end{corollary}
\begin{proof}

Clearly,
$\mathcal{F}(X) = -X + X \Delta_{G_1,X} G_1 X
 = -X\Delta_{G_1,X}$.
Following from a direct computation, we see that $\mathcal{H}(X)$ satisfies the following two properties:
\begin{eqnarray*}
\mathcal{H}(X) X \mathcal{H}(X) &=&
(I+G_1 X)^{-1} G_1X (I+G_1 X)^{-1} G_1\\
&=&(I+G_1 X)^{-1} (I- (I+G_1 X)^{-1}) G_1\\
&=&
\mathcal{H}(X) - \Delta_{X,G_1}^H G_1 \Delta_{X,G_1} \leq \mathcal{H}(X),\\
\mathcal{H}(Y)-\mathcal{H}(X)
&=&
\mathcal{H}(X)(X-Y)\mathcal{H}(Y).
\end{eqnarray*}
Note that
\begin{eqnarray*}
H_1 & = &X - A_1^H X\Delta_{G_1,X}A_1 \\
& = & X- A_1^H X A_1 +
A_1^H X G_1 \Delta_{X,G_1} XA_1,
\end{eqnarray*}
and
\[
\widehat{T}_X =
{(I-(I+G_1X)^{-1} G_1 X) A_1  = }A_1 - \mathcal{H}(X)XA_1.
\]
Thus, part (c) follows directly from  Remark~\ref{SayedRM}, which completes the proof.

%
%
%

\end{proof}

Based on Theorem~\ref{SayedRM}, we have the condition of the existence of a unique positive definite solution of~\eqref{eq:NMEM}.

\begin{theorem} \label{thm23}
Let $G_1$ and $H_1$
be two matrices defined by~\eqref{G1} and~\eqref{H1} with minus signs, and let $P_2$ be the set in~\eqref{SetP2}.
\begin{itemize}
\item [\rm{(a)}]
If $H_1>0$, then
there exists a positive definite matrix $X$ in $P_2$ such that $X$ is also a solution of~\eqref{eq:NMEM}.

\item [\rm{(b)}]
If $G_1>0$ and $H_1 >0$, then the positive definite solution of~\eqref{eq:NMEM} exists uniquely. In particular,
%
$\widehat{T}_X := \Delta_{\overline{G},{X}} \overline{A}\Delta_{G,\overline{X}}A
 = \Delta_{G_1,X} A_1
$ and $\rho(\widehat{T}_X)<1$.

\end{itemize}
\end{theorem}
\begin{proof}
{
It is true that the set $[H_1, H] =\{
X\in\mathbb{P}_{n} | H_1\leq X\leq H
\}$ is a compact convex subset of the Banach space  {$\mathbb{C}^{n\times n}$ with an unitarily invariant matrix norm.}} Also, the operator $\mathcal{F}_-$ maps $[H_1,H]$ into itself, since
\[
H_1=\mathcal{F}_-(H)\leq \mathcal{F}_-(X) \leq H
\]
for $H_1\leq X \leq H$.  It then follows from the Schauder fixed point theorem (see, e.g.~\cite{Reurings2003}) that
$\mathcal{F}_-$ has a fixed point $X$ in $[H_1,H]$.  This implies that there exists a element $X\in P_2$ and $X$ solves~\eqref{eq:NMEM}.

Considering this solution $X$ of~\eqref{eq:NMEM},  it follows that $X$ is a solution of the equation
\begin{equation}\label{eq:2NMEM}
X = F^{(2)}_-(X) = H_1+ A_1^H X\Delta_{G_1,X}A_1.
\end{equation}
Note that
the uniqueness of the solution of~\eqref{eq:NMEM} is guaranteed, once the solution of~\eqref{eq:2NMEM} is unique. By Corollary~\ref{cor:main}, this is immediately true, since $G_1 > 0$ and $H_1 > 0$.
Also,
\begin{align*}
\widehat{T}_X&=\Delta_{\overline{G},{X}} \overline{A}\Delta_{G,\overline{X}}A\\
& {  = \Delta_{\overline{G},{X}}
\overline{A}
\left(I+G
\left(\overline{H} -
\overline{A}^H X \Delta_{\overline{G},X}\overline{A}\right)\right)^{-1}
A
}\\
%
& { = \Delta_{\overline{G},{X}}
\overline{A}
\left(\Delta_{G,\overline{H}}
+ \Delta_{G,\overline{H}} G\overline{A}^H
X
\left( I+(\overline{G} -\overline{A}
\Delta_{G,\overline{H}}G\overline{A}^H)X
\right)^{-1}
\overline{A}
\Delta_{G,\overline{H}}
\right)A
}\\
& {  = \Delta_{\overline{G},{X}}
\overline{A}
(\Delta_{G,\overline{H}}
+\Delta_{G,\overline{H}} G\overline{A}^H
X
\Delta_{G_1, X}
\overline{A}
\Delta_{G,\overline{H}}
)A
}\\
&{ =\Delta_{\overline{G},{X}}  A_1 + \Delta_{\overline{G},{X}}  \overline{A}\left(\Delta_{{G},\overline{H}}G \overline{A}^H X\Delta_{G_1,X} \overline{A}\Delta_{{G},\overline{H}} \right) A} \\
&{=\Delta_{\overline{G},{X}} \left(I+G_1 X +  \overline{A}\Delta_{{G},\overline{H}}G \overline{A}^H X \right)\Delta_{G_1,X} A_1} \\
&=\Delta_{G_1,X} A_1.
\end{align*}
By Corollay~\ref{cor:main},
$\rho(\widehat{T}_X)  <1$, since
\[
\widehat{T}_X    =\Delta_{G_1,X} A_1=
{(I-(I+G_1X)^{-1} G_1 X) A_1  = }A_1 - \mathcal{H}(X)XA_1,
\]
which completes the proof.

\end{proof}

}


\section{Iterative method and convergence analysis}\label{sec:IM}


In this section, a method originated from the fixed point iteration will be presented to solve~\eqref{eq:NME} indirectly. A direct method to solve~\eqref{eq:NME} is referred to appendix~\ref{Sec:fix} for the details. We show that our proposed approach can give rise to an accelerated way with the rate of \rm{r}-superlinear convergence up to any desired order in Section~\ref{Sec:acc}.

Let $\mathcal{R}(X) = H_1 + A_1^HX\Delta_{G_1,X}A_1$ represent the computation of the right-hand side of~\eqref{eq:DAREs}, and let $X_d$ be a solution of
 ~\eqref{eq:DAREs}, that is,
 \[
 X_d = \mathcal{R}(X_d).
 \]
 Following from a similar derivation for~\eqref{eq:DAREs}, it can be seen that
 \[
X_d = \mathcal{R}(\mathcal{R}(X_d))
= H_2 + A_2^H X_d \Delta_{G_2,X_d}A_2,
\]
where
$A_2 = A_1\Delta_{G_1,{H_1}}A_{1}$,
$G_2 = G_1+A_1\Delta_{G_{1},{H_1}}G_{1}A_1^H$, and
$H_2 = H_{1}+A_{1}^H H_{1}\Delta_{G_{1},H_{1}}A_{1}$.
Continually, we have
 \[
X_d = \mathcal{R}^{(k-1)}(\mathcal{R}(X_d))
= H_{k} + A_{k}^H X_d \Delta_{G_k,X_d}A_k,
\]
where $A_k$, $G_k$, and $H_k$ for $k = 1, 2,\ldots,$ be three matrices denoted by
\begin{subequations}\label{fix}
\begin{eqnarray}
A_k &=&
A_1\Delta_{G_{k-1},{H_1}}A_{k-1},\\
G_k &=&
G_1+A_1\Delta_{G_{k-1},{H_1}}G_{k-1}A_1^H,\\
H_k &=&
H_{k-1}+A_{k-1}^H H_1\Delta_{G_{k-1},{H_1}}A_{k-1},
\end{eqnarray}
\end{subequations}
with initial matrices $G_1$, $H_1$, and $A_1$ defined by~\eqref{G1},~\eqref{H1},
and~\eqref{A1}, respectively.
Note that the iterative method given by~\eqref{fix} provide a direct way to solve~\eqref{eq:DAREs} and an indirect way to solve~\eqref{eq:NME}.
%
We show in the next result that~\eqref{fix} has a semigroup property. Its proof is quite lengthy, though it is done by mathematical induction. To the reader's interest, we put the proof in the appendix~\ref{sec:trans}.

\begin{theorem}\label{trans}
If all sequences of matrices generated by \eqref{fix} are well-defined, then the sequence $(A_k,G_k,H_k)$ satisfies the following property:
\begin{subequations}\label{tran1}
\begin{align}
A_{i+j}&=A_{j}(I+G_{i}H_{j})^{-1}A_{i},\\
G_{i+j}&=G_{j}+A_{j}(I+G_{i}H_{j})^{-1}G_{i}(A_{j})^H,\\
H_{i+j}&=H_{i}+(A_{i})^H H_{j}(I+G_{i}H_{j})^{-1}A_{i},
\end{align}
\end{subequations}
for all integers $i,j\geq 1$.
\end{theorem}
{
Based on Theorem~\ref{trans}, we have $H_k =
H_{1}+A_{1}^H H_{k-1}\Delta_{G_{1},H_{k-1}}A_{1}$. Hence,
the iteration in~\eqref{fix} is called the fixed point iteration, since its purpose is to construct a convergent sequence $H_k$ to solve~\eqref{eq:DAREs}.
}

From Theorem~\ref{thm:pd1a}, we know that if the coefficient matrix $A$ satisfies $\rho(\overline{A}A)< 1$, then the set $P_1$
is nonempty and the positive definite solution of~\eqref{eq:NMEP} uniquely exists. Our next result is to prove that the sequence of $(A_k, G_k, H_k)$ in~\eqref{fix} is well-defined, and $H_k$ tends to this positive definite solution.

\begin{Lemma}\label{Lem1}
Let $A, G, H \in \mathbb{C}^{n \times n}$ and $G, H >0$ be coefficient matrices in~\eqref{eq:NMEP}.
Then,

\begin{itemize}
\item [\rm{(a)}] $(A_k,Q_k,H_k)$ is well-defined for all integers $k\geq 1$. 

 \item [\rm{(b)}]
 If $X\in P_1$, then $X$ is an upper bound of $\{H_k\}$. In particular,
 \[
X \geq H_k \geq H_{k-1} \geq \cdots \geq H_1\geq H.
\]

\item [\rm{(c)}]  If $\rho(\overline{A}A)< 1$, $H_k$ converges to the unique positive definite solution of ~\eqref{eq:NMEP} as $k\rightarrow \infty$.
 \end{itemize}
\end{Lemma}

\begin{proof}

First, the proof of part (a) is completed, once the matrix $\Delta_{G_{k-1}, H_1}$ exists for any integer $k\geq 2$.  This suffices to show that the product of any eigenvalue of $G_{k-1}$ and $H_1$
is not equal to $-1$.
From~\eqref{G1} and~\eqref{H1}, it can be seen that 
\begin{eqnarray}
G_1&=&\overline{G}^H+\overline{A}G^H\Delta_{\overline{H}^H,{G}^H}\overline{A}^H= {G_1}^H>0,
\\
H_1&=&{H}^H+{A}^H \Delta_{\overline{H}^H, {G}^H}\overline{H}^H{A} = H_1^H>0,
\end{eqnarray}
since $G, H>0$.
Similarly, we have $G_{k} = G_{k}^H > 0$
and $H_{k} = H_{k}^H > 0$
 for any integer $k \geq 2$. This implies that
$\sigma(G_{k-1}H_{1}) \subseteq \mathbb{R}^+$, since $G_{k-1} >0$ and $H_1>0$, which completes the proof of part (a). Here $\mathbb{R}^+$ is the positive real line.

Second, if there exists $X \in P_1$, then $X \geq H$.
%
%
Note that
\begin{eqnarray}\label{eq:reP1}
H_{k} 
&=& {\mathcal{F}^{(2)}_+ (H_{k-1}) =  \mathcal{F}^{(2(k-1))}_+ (H_{1})} =  \mathcal{F}^{(2k-1)}_+ (H),
\end{eqnarray}
for all integers $k\geq 2$. Thus, we have
\begin{equation}\label{eq:reP2}
 X - H_{k} \geq \mathcal{F}_+^{(2k-1)}(X)-
 \mathcal{F}^{(2k-1)}_+ (H) \geq 0,
\end{equation}
since $\mathcal{F}_+$ is an order preserving operator.
It follows from~\eqref{eq:reP1} and~\eqref{eq:reP2} that
\[
X \geq H_k \geq H_{k-1} \geq \cdots \geq H_1\geq H,
\]
which completes the proof of part (b).

Third, since $\rho(\overline{A}A) < 1$, Theorem~\ref{1ALremark}
implies that there exists $X\in P_1$. It follows from Lemma~\ref{MCT} that the sequence $\{H_k\}$ converges, i.e.
\[
H_\ast := \lim\limits_{k\rightarrow\infty} H_k
\]
exists and satisfies~\eqref{eq:DAREs}. By Theorem~\ref{thm:pd1a}, there exists a unique positive definite solution to~\eqref{eq:NMEP}. Since the solution of~\eqref{eq:NMEP} is also a solution of~\eqref{eq:DAREs}.
Provided $G_1>0$, Corollary~\ref{cor:main} implies that~\eqref{eq:DAREs} can have only one positive definite solution, which completes the proof.
%

\end{proof}

For~\eqref{eq:NMEM}, a similar result can be derived as follows. Since the proof is similar to Lemma~\ref{Lem1}, we omit our proof here.

\begin{Lemma}\label{Lem2}
For~\eqref{eq:NMEM}, let
$G_1$ and $H_1$
be matrices defined by~\eqref{G1} and~\eqref{H1} with minus signs, and let $P_2$ be the set in~\eqref{SetP2}. Suppose that
$H_1$ and $G_1 >  0$. Then,
 \begin{enumerate}
\item [\rm{(a)}] $(A_k,Q_k,H_k)$ is well-defined for all integers $k\geq 1$.

 \item [\rm{(b)}] If $X \in P_2$, then $X$ is an upper bound of $\{H_{k}\}$. In particular,
\[
H\geq X\geq H_{k+1}\geq H_{k} \geq \cdots \geq  H_1.
%
%
\]

\item [\rm{(c)}] $H_k$ tends to the unique positive definite solution of Eq.~\eqref{eq:NMEM} as $k\rightarrow\infty$.
 \end{enumerate}
\end{Lemma}

%


From Lemma~\ref{Lem1} and Lemma~\ref{Lem2}, we have the numerical behavior of the sequence $\{H_k\}$. To our interest, we would like to predict the behavior of the sequence $\{G_k\}$ in~\eqref{fix}. We thus consider the following dual matrix equations
\begin{subequations}\label{du}
\begin{eqnarray}
{X}&=&\overline{G}+\overline{A}\, \overline{X} (I+ \overline{H}\,\overline{X})^{-1} \overline{A}^H,\label{dua}\\
{X}&=&\overline{G}-\overline{A}\, \overline{X} (I+ \overline{H}\,\overline{X})^{-1} \overline{A}^H.\label{dub}
\end{eqnarray}
\end{subequations}
%
%
It has been shown in Theorem~\ref{thm:pd1a} and Theorem~\ref{thm23} that there exists a unique positive definite solution $X$ of~\eqref{eq:NME}, once certain conditions are satisfied. Here, we assume that the coefficient matrix $A$ is nonsingular and define $Y= -X^{-1}$, where $X$ is the solution of~\eqref{du}. Following from~\eqref{du}, we have
\[
A^{-1}(\overline{X}-{G})A^{-H}=\pm({X}^{-1}+ {H})^{-1}.
\]
This implies that
\[
{X}^{-1}+ {H} = \pm A^{H}(\overline{X}-{G})^{-1}A, \]
That is,
{
\[
Y =
H  {\pm} A^{H} \overline{Y}(I+{G}\overline{Y})^{-1}A,
\]
}
which is exactly equivalent to the matrix equation~\eqref{eq:NME}.
Like Theorem~\ref{thm:pd1a} and Theorem~\ref{thm23}, we thus have the following result.
%
\begin{theorem}
Assume that $A$ is nonsingular. Then,
\begin{itemize}
\item[1.]There exists a unique negative definite solution to{~\eqref{eq:NMEP}} if $\rho(A\overline{A})<1$.
\item[2.]There exists a unique negative definite solution to{~\eqref{eq:NMEM}} if $G_1>0$ and $H_1 > 0$.
\end{itemize}
\end{theorem}
%
%
%
Now, we would like to investigate the relationship between the sequence $\{G_k\}$ and the dual equations~\eqref{du}. For the sake of simplicity, let $\mathcal{G}_\pm({X})$ be the matrix operator defined by
\[
\mathcal{G}_\pm({X}) =  \overline{G}\pm\overline{A}\,\overline{X}\Delta_{\overline{H},\overline{X}}\overline{A}^H.
\]
Then, the dual equations~\eqref{du} can be rewritten as
\[
X=\mathcal{G}_\pm({X}).
\]
Analogous to the case of operator $\mathcal{F}_\pm$, we have the following formula
\[
X=\mathcal{G}_\pm^{(2)}({X})=\widetilde{H}_1+\widetilde{A}_1^H  X\Delta_{\widetilde{G}_1,X}\widetilde{A}_1,
\]
where
\begin{subequations}
\begin{align}
\widetilde{A}_1&=
\overline{\overline{A}^H} \Delta_{\overline{H},G} \overline{A}^H=A_1^H, \\
%
\widetilde{G}_1&= H \pm{A}^H \Delta_{\overline{H}, {G}}\overline{H} A =H_1,\\
\widetilde{H}_1&=\overline{G}\pm \overline{A} G \Delta_{\overline{H}, {G}}\overline{A}^H =G_1,
\end{align}
\end{subequations}
or even more,
\[
X=\mathcal{G}_\pm^{(2k)}({X})=\widetilde{H}_k+\widetilde{A}_k^H  X\Delta_{\widetilde{G}_k,X}\widetilde{A}_k,
\]
where
\begin{subequations}\label{hatAGH}
\begin{align}
\widetilde{A}_k&=
\widetilde{A}_1\Delta_{\widetilde{G}_{k-1},{\widetilde{H}_1}}\widetilde{A}_{k-1},\\
\widetilde{G}_k&=
\widetilde{G}_1+\widetilde{A}_1\Delta_{\widetilde{G}_{s-1},{\widetilde{H}_1}}\widetilde{G}_{s-1}\widetilde{A}_1^H,\\
\widetilde{H}_k&=
\widetilde{H}_{k-1} + \widetilde{A}_{k-1}^H \widetilde{H}_1 \Delta_{
\widetilde{G}_{k-1},
\widetilde{H}_1} \widetilde{A}_{k-1}.
%
\end{align}
\end{subequations}
By induction on $k$, it is true that
\begin{align}\label{dual}
\widetilde{A}_k=A_k^H,\,\widetilde{G}_k=H_k,\,\widetilde{H}_k=G_k.
\end{align}
Thus, the sequence of matrices $(\widetilde{A}_k,\widetilde{G}_k,\widetilde{H}_k)$ generated by the iterations~\eqref{fix} with initial matrices $(\widetilde{A}_1,\widetilde{G}_1,\widetilde{H}_1)=(A_1^H,H_1,G_1)$ is well-defined, once the sequence of matrices $({A}_k,{G}_k,{H}_k)$ is well-defined.
Let $D_1$ and $D_2$ be two sets defined by
\begin{eqnarray}
D_1 &=&  \{Y >0 | Y \geq\mathcal{G}_{+}(Y)\}, \\
D_2 &=&  \{Y >0 | \overline{G} \geq Y \geq\mathcal{G}_{-}(Y) \},
\end{eqnarray}
respectively.
By~\eqref{hatAGH}, we have the following result. Its proof is similar to Lemma~\ref{Lem1} and Lemma~\ref{Lem2} and is omitted here.

\begin{Lemma}\label{Lem3}
\par\noindent
Let $A, G, H \in \mathbb{C}^{n \times n}$ be the coefficient matrices of~\eqref{eq:NME} such that  $G, H >0$.  Then,

\begin{enumerate}
\item
For~\eqref{eq:NMEP},
    \begin{itemize}
        \item [\rm{(a)}] $(\widetilde{A}_k,\widetilde{G}_k,\widetilde{H}_k)$ is well-defined for all integers $k\geq 1$. 

 \item [\rm{(b)}]
 If $Y\in D_1$, then $Y$ is an upper bound of $\{G_k\}$. In particular,
 \[
Y \geq {G}_k \geq {G}_{k-1} \geq \cdots \geq {G}_1\geq \overline{G}.
\]

\item [\rm{(c)}]  If $\rho(A\overline{A})< 1$, ${G}_k$ converges to the unique positive definite solution of ~\eqref{dua} as $k\rightarrow \infty$.

 \end{itemize}

\item Assume that $G_1>0$ and $H_1>0$ .
For~\eqref{eq:NMEM},

 \begin{enumerate}

\item [\rm{(a)}] $(\widetilde{A}_k,\widetilde{G}_k,\widetilde{H}_k)$ is well-defined for all integers $k\geq 1$.

 \item [\rm{(b)}] If $Y \in D_2$, then $Y$ is an upper bound of $\{G_{k}\}$. In particular,
\[
\overline{G} \geq Y\geq {G}_{k+1}\geq {G}_{k} \geq \cdots \geq {G}_1.
%
%
\]

\item [\rm{(c)}] ${G}_k$ tends to the unique positive definite solution of ~\eqref{dub} as $k\rightarrow\infty$.
 \end{enumerate}

%
%
\end{enumerate}
\end{Lemma}

%
%
%
%

In summary, following from Lemmas~\ref{Lem1}, ~\ref{Lem2}, and~\ref{Lem3}, we have the following main result of this section.
%
%
\begin{theorem}\label{main1}
Let $A, G, H \in \mathbb{C}^{n \times n}$ be the coefficient matrices of~\eqref{eq:NME} such that  $G, H >0$. Consider the sequence of matrices $(A_k,G_k,H_k)$  generated by iterations \eqref{fix} with a given  initial matrices $(A_1,G_1,H_1)$ defined by \eqref{A1}, \eqref{G1}, and \eqref{H1}, respectively.
Let $H_{\infty}=\lim\limits_{\ell\rightarrow\infty}H_{\ell}$ and $G_{\infty}=\lim\limits_{\ell\rightarrow\infty}G_{\ell}$.
Then,
\begin{enumerate}
  \item Assume that $\rho(A\overline{A})<1$. For~\eqref{eq:NMEP},
\begin{itemize}
\item[\rm{(a)}]$H_{\infty}$ is the unique positive definite solution to~\eqref{eq:NMEP}.

 \item[\rm{(b)}]$-G_{\infty}^{-1}$ is the unique negative definite solution to~\eqref{eq:NMEP} if $A$ is nonsingular.
\end{itemize}

  \item Assume that $H_1>0$ and $G_1>0$. For~\eqref{eq:NMEM},
  \begin{itemize}
\item[\rm{(a)}] $H_{\infty}$ is the unique positive definite solution to~\eqref{eq:NMEM}.
 \item[\rm{(b)}] $-G_{\infty}^{-1}$ is the unique negative definite solution to~\eqref{eq:NMEM} if $A$ is nonsingular.
\end{itemize}
\end{enumerate}
\end{theorem}
%
%
%

\begin{Remark}
It is interesting to ask whether the matrix $Y= -X^{-1}$, where $X$ is the solution of~\eqref{du}, is still a negative positive solution of~\eqref{eq:NME} if $A$ is singular.
To answer this question, we see that
\begin{align*}
I+G\overline{Y}&=I-G({G}\pm{A} {X} (I+ {H}{X})^{-1} {A}^H)^{-1}\\
&=I-G(G^{-1}\mp G^{-1}A X
(I+H X \pm A^H G^{-1} A X)^{-1}A^H G^{-1})\\
&=\pm AX(I+(H\pm A^H G^{-1} A)X)^{-1}A^H G^{-1}.
\end{align*}
Namely, rank$(I+G\overline{Y})$=rank$(A)$. We conclude that the matrix $Y=-X^{-1}$ is not a solution of~\eqref{eq:NME} when $A$ is singular, since $I+G\overline{Y}$ is not invertible.
\end{Remark}

%
%
%
%
%
%
%

\section{An acceleration of iterative method}\label{Sec:acc}

Let $\{A_k, G_k, H_k\}$ be the sequence of matrices generated by~\eqref{fix}. It has been shown in Theorem~\ref{trans} that matrices $A_k$, $G_k$, and $H_k$, for each $k$, depend only on the subscripts in $A_i$, $A_j$, $G_i$, $G_j$, $H_i$, and $H_j$, once $i+j =k$. Our next algorithm is to fully take advantage of this invariance to design an algorithm with speed of convergence of any desired order.

\begin{Algorithm} \label{aa2}
{\emph{(An accelerated iteration method to solve~\eqref{eq:NME})}}
\begin{enumerate}
\item
Given a positive integer $r>1$,
%
let $(\widehat{A}_0,\widehat{G}_0,\widehat{H}_0)=(A_1,G_1,H_1)$ with initial matrices $G_1$,  $H_1$, and $A_1$ defined by~\eqref{G1},~\eqref{H1},
and~\eqref{A1}, respectively;
\item    {For} $k=1,2,\ldots$, {iterate}
\begin{align*}
\widehat{A}_k&:={A}_{k-1}^{(r-1)}(I_n+\widehat{G}_{k-1}{H}_{k-1}^{(r-1)})^{-1}\widehat{A}_{k-1},\\
\widehat{G}_k&:={G}_{k-1}^{(r-1)}+{A}_{k-1}^{(r-1)}(I_n+\widehat{G}_{k-1}{H}_{k-1}^{(r-1)})^{-1}\widehat{G}_{k-1}({A}_{k-1}^{(r-1)})^H,\\
\widehat{H}_k&:=\widehat{H}_{k-1}+\widehat{A}_{k-1}^H {H}_{k-1}^{(r-1)}(I_n+\widehat{G}_{k-1}{H}_{k-1}^{(r-1)})^{-1}\widehat{A}_{k-1},
\end{align*}
    until convergence (see Section~\ref{sec:NE} for example), where the sequence $({A}_{k-1}^{(r-1)},{G}_{k-1}^{(r-1)},{H}_{k-1}^{(r-1)})$ is defined in step 3.
\item
     {For} $\ell=1,\cdots,r-2$, iterate
\begin{align*}
{A}_{k-1}^{(\ell+1)}&:={A}_{k-1}^{(\ell)}(I_n+\widehat{G}_{k-1}{H}_{k-1}^{(\ell)})^{-1}\widehat{A}_{k-1},\\
{G}_{k-1}^{(\ell+1)}&:={G}_{k-1}^{(\ell)}+{A}_{k-1}^{(\ell)}(I_n+\widehat{G}_{k-1}{H}_{k-1}^{(\ell)})^{-1}\widehat{G}_{k-1}({A}_{k-1}^{(\ell)})^H,\\
{H}_{k-1}^{(\ell+1)}&:=\widehat{H}_{k-1}+\widehat{A}_{k-1}^H {H}_{k-1}^{(\ell)}(I_n+\widehat{G}_{k-1}{H}_{k-1}^{(\ell)})^{-1}\widehat{A}_{k-1},
\end{align*}
with $({A}_{k-1}^{(1)},{G}_{k-1}^{(1)},{H}_{k-1}^{(1)})=(\widehat{A}_{k-1},\widehat{G}_{k-1},\widehat{H}_{k-1})$.
 \end{enumerate}
\end{Algorithm}

By Theorem~\ref{trans}, we have the following result. Its proof is  straightforwardly done by induction. We thus omit the proof here.

\begin{Remark}
If $(A_k, G_k, H_k)$ for all integers $k\geq 1$ is well-defined, that
\[
(\widehat{A}_k,\widehat{G}_k,\widehat{H}_k) =
 ({A}_{r^{k}},{G}_{r^{k}},{H}_{r^{k}})
 \]
  for all integers $k\geq 1$.
\end{Remark}
The convergence analysis of Algorithm~\ref{aa2} can be done by means of the following properties.  Since the proof is long and tedious, we put it in Appendix~\ref{sec:3}.


\begin{Lemma}\label{lem:conv}
Assume that $(A_k,G_k,H_k)$ is a well-defined sequence of matrices from \eqref{fix} and this sequence is convergent. Let
\begin{eqnarray}\label{eq:HGTS}
H_{\infty}&=&\lim\limits_{k\rightarrow\infty}H_{k},\quad G_{\infty}=\lim\limits_{k\rightarrow\infty}G_{k},\nonumber\\
T_k &=& \Delta_{G_k,H_\infty} A_k,\quad S_k = A_k\Delta_{G_\infty,H_k},
\end{eqnarray}
for  all integers $k\geq 1$.
Then, the following three conditions are satisfied.

\par\noindent
\begin{enumerate}

\item  $
T_k= T_1^k$ and $
S_k =  S_1^k$.

\item
$H_\infty-H_k=T_k^H H_\infty A_k=T_k^H (H_\infty^{-1}+G_k) T_k$ and
$G_\infty-G_k=S_k G_\infty A_k^H=S_k (G_\infty^{-1}+H_k)S_k^H$.

\item $
\sigma(T_1)=\sigma(S_1^H).
$
\end{enumerate}
\end{Lemma}

Let all the sequences in Algorithm~\ref{aa2} be well-defined. Our next result is to show that once $\rho({T}_1)<1$, the convergence speed of $(\widehat{A}_k,\widehat{G}_k,\widehat{H}_k)$ is \rm{r}-superlinearly with order $r$, for any integer $r > 0$. The definition of \rm{r}-superlinear convergence is referred to~\cite[Definition 4.1.3.]{Kelly1995}.

\begin{theorem}\label{thm:numite}
Suppose that $\{\widehat{A}_k,\widehat{G}_k,\widehat{H}_k\}$ is the sequence of matrices generated by iterations \eqref{fix} and be well-defined and convergent. Let
$H_{\infty}, G_{\infty}$ and $T_k, S_k$, for all integers $k\geq 1$,
be matrices defined by~\eqref{eq:HGTS}. Then,
\begin{align*}
&\limsup\limits_{k\rightarrow\infty}\sqrt[r^k]{\| \widehat{A}_k\|}\leq
\rho({T}_1),\quad
\limsup\limits_{k\rightarrow\infty}\sqrt[r^k]{\| G_\infty - \widehat{G}_k\|}\leq
\rho(T_1)^2,\\
&\limsup\limits_{k\rightarrow\infty}\sqrt[r^k]{\| H_\infty - \widehat{H}_k\|}\leq
\rho(T_1)^2.
\end{align*}


\end{theorem}

\begin{proof}
From Lemma~\ref{lem:conv}, we know that
$\widehat{A}_k
= A_{r^k} = (I+G_{r^k}H_\infty) T_1^{r^k}
$, $
G_\infty  - \widehat{G}_k
= {S}_{r^k}(G_\infty^{-1}+H_{r^k}){S}_{r^k}^H
$, and $H_\infty- \widehat{H}_k =T_{r^k}^H (H_\infty^{-1}+G_{r^k}) T_{r^k}$
. It follows that
\begin{eqnarray*}
&&\limsup\limits_{k\rightarrow\infty}\sqrt[r^k]{\| \widehat{A}_k\|}
=
\limsup\limits_{k\rightarrow\infty}\sqrt[r^k]{
G_{r^k}(G_{r^k}^{-1}+H_\infty) T_1^{r^k}}\\
&&\leq \limsup\limits_{k\rightarrow\infty}\sqrt[r^k]{\| G_\infty\|}\cdot \limsup\limits_{k\rightarrow\infty}\sqrt[r^k]{\| G_1^{-1}+ H_\infty\|}
\cdot\limsup\limits_{k\rightarrow\infty}\sqrt[r^k]{\| {T}_1^{r^k}\|}=\rho({T}_1),\\
&&\limsup\limits_{k\rightarrow\infty}\sqrt[r^k]{\| G_\infty - \widehat{G}_k\|}\leq
\limsup\limits_{k\rightarrow\infty}\sqrt[r^k]{\| S_1^{r^k}\|}\cdot
\limsup\limits_{k\rightarrow\infty}\sqrt[r^k]{\|(G_1^{-1}+H_{\infty}) \|}\cdot\\
&&
\limsup\limits_{k\rightarrow\infty}\sqrt[r^k]{\| S_1^{r^k}\|}=
\rho(T_1)^2,\\
&&\limsup\limits_{k\rightarrow\infty}\sqrt[r^k]{\| H_\infty - \widehat{H}_k\|}\leq
\limsup\limits_{k\rightarrow\infty}\sqrt[r^k]{\| {T}_{1}^{r^k}\|}\cdot
\limsup\limits_{k\rightarrow\infty}\sqrt[r^k]{\|(H_1^{-1}+G_{\infty}) \|}\cdot\\
&&\limsup\limits_{k\rightarrow\infty}\sqrt[r^k]{\| {T}_1^{r^k}\|}=
\rho(T_1)^2.
\end{eqnarray*}
Here, the last equalities follow from the well-known Gelfand's formula such that for any matrix norm $\|\cdot\|$, we have $\rho(A)
= \limsup\limits_{k\rightarrow\infty} \|A^k\|^{1/k}
$.

\end{proof}

\section{Numerical experiments}\label{sec:NE}
Under the assumptions of Theorem~\ref{main1}, two numerical examples are used in this section to demonstrate the application of accelerated techniques given by Algorithm~\ref{aa2}.
We compare Algorithm~\ref{aa2} with the standard fixed point iterations:
\begin{align}\label{fix1}
X_{k+1}=\mathcal{F}_\pm(X_k), \quad \mbox{with } X_1=H.
\end{align}
It can be shown that the convergence speed of \eqref{fix1} is r-linearly if $\rho({T}_1)<1$. The details for the convergence analysis can be found in Appendix~\ref{Sec:fix}. For clarity, two things should be  emphasized here.
First, the unique negative definite solution of~\eqref{eq:NME} can be obtained by Algorithm~\ref{aa2} when $A$ is nonsingular. That is, Algorithm~\ref{aa2}  enable us to solve the unique positive and negative definite solutions, simultaneously. Second, when $\rho({T}_1)\approx 1$, then iteration~\eqref{fix1} could be very slow. However, this disadvantage can be overcome without any difficulty by Algorithm~\ref{aa2}. While solving~\eqref{eq:NME}, we show  that the use of Algorithm~\ref{aa2} tends to has less computational time and higher accuracy than the fixed point methods given by~\eqref{fix1}.

All computations were performed using MATLAB/version 2016b on MacBook Air with a 2.2 GHZ Intel Core i7 processor and 8 GB of memory. To gauge the effectiveness of  our algorithm, we employ the parameters, residual (Res) and the normalized residual (NRes) with definitions defined below:
\begin{align*}
&\mbox{Res} : = \| X-\mathcal{F}_\pm(X)\|_F,\\
&\mbox{NRes} : = \frac{\| X-\mathcal{F}_\pm(X)\|_F}{\|H\|_F+\|A\|^2_F\|X\|_F\|\Delta_{G,X}\|_F},
\end{align*}
where $X$ is an approximate maximum positive solution to
\eqref{eq:NME}. All iterations are terminated whenever Res or NRes is less than {or equal to} $n\mathbf{u}$, where $\mathbf{u}=2^{-52}\cong 2.22 \times 10^{-16}$ is the machine zero.

\begin{example}~\label{ex1}
Let  $n=100$ and $\widehat{G}, \widehat{H} \in \mathbb{R}^{n \times n}$ be two real diagonal  matrices with given positive diagonal elements between $0$ and $1$. They are then reshuffled by the unitary matrix $Q \in \mathbb{C}^{n \times n}$ to form
\begin{equation}\label{eq:GH}
(G,H) = (Q^H \widehat{G} Q, Q^H \widehat{H} Q),
\end{equation}
that is, in \texttt{MATLAB} commands, we define
\begin{eqnarray*}
\widehat{G} &=&1e2*diag(rand(n)), \,\,
\widehat{H}=1e2*diag(rand(n)), \\
Q&=&orth(crandn(n)).
\end{eqnarray*}

For~\eqref{eq:NMEP}, Theorem~\ref{thm:pd1a} implies that a unique positive definite solution exists, if $\rho(\overline{A} A)<1$. To satisfy this constraint, let $\widehat{A}$ be a randomly generated square complex matrix, let $a$ be a random number lying in the interval $(0,1)$, and let $temp$ be the spectral radius of $\widehat{A}^H \widehat{A}$, namely,
 \begin{eqnarray*}
\widehat{A}&=&crandn(n), \,\,a=rand,\\
temp &=&max(abs(eig(conj(\widehat{A})*\widehat{A}))).
\end{eqnarray*}
We then have a matrix
\begin{align}\label{eq:AA}
A=\sqrt{a}*\widehat{A}/\sqrt{temp}
\end{align}
satisfying $\rho(\overline{A} A)<1$ so that the unique positive definite solution to~\eqref{eq:NMEP} exists.

For~\eqref{eq:NMEM}, we have shown that the unique positive definite solution exists if $G_1>0$ and $H_1>0$. To this end, we repeatedly generate matrices $A$, $G$, and $H$ by~\eqref{eq:GH} and~\eqref{eq:AA}
until $G_1$ and $H_1>0$ are satisfied. We record numerical results in Table~\ref{table:cputime1}.

Note that in Tables~\ref{table:cputime} and~\ref{table:cputime1}, the values in the second row are the results obtained using the standard fixed point method given in~\eqref{fix1},  and the values in the other rows are results obtained using Algorithm~\ref{aa2} with $r=2,3,4,5$, respectively.
The minimal number of iterations (\emph{MinIt}), the maximal number of iterations (\emph{MaxIt}), the average number of iterations (\emph{AveIt}), and the average elapsed times of iterations (\emph{AveTime}) performed by the fixed point method and our algorithm are recorded by choosing 100 initial matrices $(G,H,A)$ randomly, as are described above.
Let $N_1$ and $N_r$, with $r = 2,3,4,5$, be the least integer numbers satisfying
\[
\rho{(T_1)}^{N_1} < n\cdot \mathbf{u}, \quad \left (\rho{(T_1)}^2\right)^{r^{N_r}} < n\cdot\mathbf{u}, \quad r = 2,3,4,5,
\]
respectively. That is, $N_1$ and $N_r$, with $r= 2,3,4,5$, are integer numbers  defined by
\begin{equation}\label{eq:N1}
N_1= \left [
\frac{
\log_{10}(n\cdot\mathbf{u})}
{
\log_{10}\left({\rho({T}_1)}\right)
}
\right ]+1
\end{equation}
and
\begin{equation}\label{eq:Nr}
N_r= \left [\log_{r}\left(
\frac{
\log_{10}(n\cdot\mathbf{u})}{\log_{10}\left({\rho({T}_1)^2}\right)
}
\right )
\right
]+1.
\end{equation}
Here, the symbol  $[x]$ denotes the floor of $x$, i.e., the largest integer less than or equal to $x$ and $T_1 = \Delta_{G_1,H_\infty} A_1$.
We then record in the fifth column of Tables~\ref{table:cputime} and~\ref{table:cputime1} the number of iterations (\emph{TheIt}) estimated by means of~\eqref{eq:N1} and~\eqref{eq:Nr}.
The records show that  the estimated numbers~\emph{TheIt} are highly correlated to the numerical iterative numbers~\emph{AveIt}. This implies that in practice, \emph{TheIt} can be served as a priori prediction of the possible iterative numbers. Also, we can see from the records in the columns of \emph{AveIt} and \emph{TheIt} that our algorithm outperform the fixed point method not only in the number of required iterations, but also in the elapsed times.

%
%
%


\begin{table}[h!!!]
  \caption{Numerical experiments by means of the fixed point method (F.P.) given in~\eqref{fix1}
  and accelerated methods originated from Algorithm (Alg.)~\ref{aa2} ($r=2,3,4,5$) to solve~\eqref{eq:NMEP}}\label{table:cputime}
  \centering
  \begin{tabular}{cccccc}\hline
    Method & {MinIt} &  {MaxIt} & {AveIt} & {TheIt} & AveTime\\
\hline
     F.P. with ``+'' & 1 &  7& 3.43 & 3.95 &3.8526e-02\\
   Alg. \ref{aa2} with $r=2$ & 1 & 2 &   1.41 & 1.23&2.1181e-02\\
   Alg. \ref{aa2} with $r=3$ &  1 & 2 & 1.04  &1.02 &1.9381e-02\\
  Alg. \ref{aa2} with $r=4$ & 1 & 1 &  1& 1            &2.1875e-02\\
  Alg. \ref{aa2} with $r=5$ & 1 & 1 &  1 & 1           &2.4443e-02\\
        \hline
  \end{tabular}
\end{table}

\begin{table}[h!!!]
  \caption{Numerical experiments by means of the fixed point method (F.P.) given in~\eqref{fix1}
  and accelerated methods originated from Algorithm (Alg.)~\ref{aa2} ($r=2,3,4,5$) to solve~\eqref{eq:NMEM}}\label{table:cputime1}
  \centering
  \begin{tabular}{cccccc}\hline
    Method & {MinIt} &  {MaxIt} & {AveIt} & {TheIt} & AveTime\\
\hline
     F.P. with ``-'' & 1 & 10& 3.8 & 4.35  &4.2394e-02\\
      Alg. \ref{aa2} with $r=2$ & 1 & 3 &  1.58 & 1.38 &2.3283e-02\\
     Alg. \ref{aa2} with $r=3$ &  1 & 2 &  1.08  & 1.05  &1.9851e-02\\
     Alg. \ref{aa2} with $r=4$ & 1 & 2 &  1.02 & 1.01  &2.1808e-02\\
     Alg. \ref{aa2} with $r=5$ & 1 & 1 &  1 & 1            &2.4341e-02\\
        \hline
  \end{tabular}
\end{table}

\end{example}

In the next example, we show that as the value of $\rho{(T_x)}$ come closer to 1, the fixed point method will fail to converge, but our algorithm can converge with no difficulty.

\begin{example}
If $n = 1$, the corresponding equations of~\eqref{eq:NME} become to
\begin{eqnarray}~\label{ex2}
     x &=& h \pm \frac{|a|^2 \overline{x}}{1+g\overline{x}},
 \end{eqnarray}
where  $a\in\mathbb{C}$ and the real numbers $g, h>0$.  To measure performance of different methods, four cases, i.e., $\rho(T_1) = \frac{1}{2}, \frac{1}{\sqrt{2}}, \frac{\sqrt{3}}{2}, \sqrt{0.9999}$ with different parameters will be taken into account; namely,
we set $g = 1$ and $a =\frac{1}{\sqrt{2}}, \frac{\sqrt{3}}{2}, \sqrt{0.9999}$, and $\sqrt{0.99999}$ corresponding to $\rho({T}_1) =\frac{1}{2}, \frac{1}{\sqrt{2}}, \frac{\sqrt{3}}{2}$ and $\sqrt{0.9999}$, respectively. For these parameters $a$
and $\rho(T_1)$, there exists a unique positive definite solution $x$ of~\eqref{ex2}  decided by
\begin{eqnarray*}
x &=& \frac{a}{\rho({T}_1)}-1>0,
\end{eqnarray*}
since $\rho({T}_1)=(1+x)^{-1} a$ and $|a|^2 < 1$. Thus, the resulting parameter
\[
h= x+ \frac{|a|^2 \overline{x}}{1+g\overline{x}} > 0 \quad (\mbox{or }
h= x- \frac{|a|^2 \overline{x}}{1+g\overline{x}} > 0)
\]
satisfies the constraint for~\eqref{eq:NME}. Also, for the minus case, i.e.,  $x = h - \frac{|a|^2 \overline{x}}{1+g\overline{x}}$, we have
\begin{eqnarray*}
h_1 = h - \frac{|a|^2 \overline{h}}{1+g \overline{h}} > 0,\\
g_1 = \overline{g} - \frac{|a|^2 g}{1+g \overline{h}} > 0.
\end{eqnarray*}
Under conditions of Theorems~\ref{thm:pd1a} and~\ref{thm23} we  see that there only exists a unique positive definite solution for both cases of~\eqref{ex2}.

 In Tables~\ref{tab22} and~\ref{tab3}, the values in the second row, $r=1$, are the results obtained using the fixed point method, and the values in the other rows are results obtained using Algorithm~\ref{aa2} with $r=2,3,4,5$, respectively. The number of iterations (\emph{Its}), the output residual (\emph{Res}), and the elapsed times of iterations (\emph{Time}) performed by the fixed point method and our algorithm are recorded correspondingly.

Table~\ref{tab22} shows that even with $10000$ steps, the solution
obtained from the fixed point method can only have accuracy up to $10^{-13}$. What is worse, Table~\ref{tab3}  shows that the fixed point method
can hardly solve~\eqref{ex2} with minus sign, even after $10000$ steps.
The residuals and elapsed times in Table~\ref{tab22} show that our accelerated technique can solve~\eqref{ex2} more accurately and efficiently.
%
Also, the number of iterations  by the fixed point method increase dramatically, while those by our accelerated techniques only has a small increase. This implies that our algorithm could provide a more reliable way to obtain numerical  solutions, even if the extreme case, i.e., $\rho(T_1)\approx 1$,  is encountered.

%

\begin{table}[htp]
\caption{The ITs, Res and Time for the problem
$x =  h + \frac{|a|^2 \overline{x}}{1+\overline{x}}$.}\label{tab22}
\begin{center}
\begin{tabular}{c|c|c|c|c|c|}
\cline{2-6}
&$\rho(T_1)$  & ${1}/{2}$ &  ${1}/{\sqrt{2}}$ & ${\sqrt{3}}/{2}$ &$\sqrt{0.9999}$
\\\hline
\multicolumn{1}{|c|}{\multirow{3}{*}{ F.P. with ``+'' }} &
\multicolumn{1}{|c|}{Its}  & 25 & 49 & 116 & *($>$10000)     \\ 
\multicolumn{1}{|c|}{}                        &
\multicolumn{1}{|c|}{Res}  & 8.3267e-17 & 1.6653e-16 & 1.9429e-16  &3.7124e-13     \\ 
\multicolumn{1}{|c|}{}                        &
\multicolumn{1}{|c|}{Time}  & 1.3828e-02 & 9.2210-03  & 1.1571e-02  &  7.9136   \\ \cline{1-6}
\multicolumn{1}{|c|}{\multirow{3}{*}{Alg. \ref{aa2} with $r=2$}} &
\multicolumn{1}{|c|}{Its}  & 4 & 5 & 6 & 17     \\ 
\multicolumn{1}{|c|}{}                        &
\multicolumn{1}{|c|}{Res}  & 2.7756e-17&  2.7756-17 & 0 &2.7105e-20      \\ 
\multicolumn{1}{|c|}{}                        &
\multicolumn{1}{|c|}{Time}  & 1.1086e-02 & 7.7436e-03 &  8.9338e-03 & 5.7342e-03   \\ \cline{1-6}
\multicolumn{1}{|c|}{\multirow{3}{*}{Alg. \ref{aa2} with $r=3$}} &
\multicolumn{1}{|c|}{Its}  & 3 & 3 & 4&11     \\ 
\multicolumn{1}{|c|}{}                        &
\multicolumn{1}{|c|}{Res}  & 5.5511e-17  & 2.7756e-17  & 0 &  0   \\ 
\multicolumn{1}{|c|}{}                        &
\multicolumn{1}{|c|}{Time}  & 4.6519e-03 & 2.8196e-03  &  4.3291e-03 &   3.2251e-03   \\ \cline{1-6}
\multicolumn{1}{|c|}{\multirow{3}{*}{Alg. \ref{aa2} with $r=4$}} &
\multicolumn{1}{|c|}{Its}  & 2 & 3 & 3 & 9     \\ 
\multicolumn{1}{|c|}{}                        &
\multicolumn{1}{|c|}{Res}  & 5.5511e-17 & 2.7756e-17  & 0 &  0   \\ 
\multicolumn{1}{|c|}{}                        &
\multicolumn{1}{|c|}{Time}  & 5.4670e-04    &   4.6242e-04 &   6.5906e-04 &  4.3089e-04   \\ \cline{1-6}
\multicolumn{1}{|c|}{\multirow{3}{*}{Alg. \ref{aa2} with $r=5$}} &
\multicolumn{1}{|c|}{Its}  & 2 & 2 & 3 & 8     \\ 
\multicolumn{1}{|c|}{}                        &
\multicolumn{1}{|c|}{Res}  & 2.7756e-17& 1.6653e-16  & 0 &0     \\ 
\multicolumn{1}{|c|}{}                        &
\multicolumn{1}{|c|}{Time}  & 3.7408-04 & 3.1476e-04 & 5.8860e-04 &  3.7914e-04    \\ \cline{1-6}
\hline
\end{tabular}
\end{center}
\end{table}

\begin{table}[htp]
\caption{The ITs, Res and Time for the problem
$x =  h - \frac{|a|^2 \overline{x}}{1+\overline{x}}$.}
\begin{center}
\begin{tabular}{c|c|c|c|c|c|}
\cline{2-6}
&$\rho(T_1)$  & ${1}/{2}$ &  ${1}/{\sqrt{2}}$ & ${\sqrt{3}}/{2}$ &$\sqrt{0.9999}$
\\\hline
\multicolumn{1}{|c|}{\multirow{3}{*}{F.P. with ``-''}} &
\multicolumn{1}{|c|}{Its}  & 25 & 50 & 120 & *($>$100000)     \\ 
\multicolumn{1}{|c|}{}                        &
\multicolumn{1}{|c|}{Res}  & 1.3878e-16 & 1.9429e-16 & 2.2204e-16 &4.0837e-09     \\ 
\multicolumn{1}{|c|}{}                        &
\multicolumn{1}{|c|}{Time}  & 8.6382e-03 &  6.8811e-03  & 8.9200e-03  &   8.3106
  \\ \cline{1-6}
\multicolumn{1}{|c|}{\multirow{3}{*}{Alg. \ref{aa2} with $r=2$}} &
\multicolumn{1}{|c|}{Its}  & 4 & 5 & 6 & 18     \\ 
\multicolumn{1}{|c|}{}                        &
\multicolumn{1}{|c|}{Res}  & 2.7756e-17& 5.5511e-17& 8.3267e-17&1.5491e-17      \\ 
\multicolumn{1}{|c|}{}                        &
\multicolumn{1}{|c|}{Time}  & 4.4193e-03 & 9.6181e-03 &     6.4900e-03& 7.6470e-03      \\ \cline{1-6}
\multicolumn{1}{|c|}{\multirow{3}{*}{Alg. \ref{aa2} with $r=3$}} &
\multicolumn{1}{|c|}{Its}  & 3 & 3 & 4& 11     \\ 
\multicolumn{1}{|c|}{}                        &
\multicolumn{1}{|c|}{Res}  & 5.5511e-17  & 5.5511e-17  & 2.7756e-17 &  1.7171e-17     \\ 
\multicolumn{1}{|c|}{}                        &
\multicolumn{1}{|c|}{Time}  &  4.7336e-03 &  3.9281e-03  &  2.9638e-03       &   3.3373e-03     \\ \cline{1-6}
\multicolumn{1}{|c|}{\multirow{3}{*}{Alg. \ref{aa2} with $r=4$}} &
\multicolumn{1}{|c|}{Its}  & 2 & 3 & 3 & 9     \\ 
\multicolumn{1}{|c|}{}                        &
\multicolumn{1}{|c|}{Res}  & 2.7756e-17 & 0  & 8.3267e-17 &   1.7362e-17     \\ 
\multicolumn{1}{|c|}{}                        &
\multicolumn{1}{|c|}{Time}  & 4.3161e-04    &   5.1107e-04 &   4.7924e-04       &  4.2425e-04   \\ \cline{1-6}
\multicolumn{1}{|c|}{\multirow{3}{*}{Alg. \ref{aa2} with $r=5$}} &
\multicolumn{1}{|c|}{Its}  & 2 & 3 & 3 & 8     \\ 
\multicolumn{1}{|c|}{}                        &
\multicolumn{1}{|c|}{Res}  & 5.5511e-17& 0 & 8.3267e-17 & 4.1064e-18      \\ 
\multicolumn{1}{|c|}{}                        &
\multicolumn{1}{|c|}{Time}  & 3.2604e-04 & 5.4106e-04 & 3.9956e-04&    3.9606e-04     \\ \cline{1-6}
\hline
\end{tabular}
\end{center}
\label{tab3}
\end{table}

\end{example}

\section{Conclusion}\label{sec:con}
In this paper, we propose sufficient conditions for the existence of a unique positive 
definite solution of~\eqref{eq:NME}. Note that an intuitive way to solve~\eqref{eq:NME} is to apply the fixed point method. Though this method is guaranteed to converge, the convergence rate tends to be slow. Numerically, we provide an accelerated way to speed up the entire iteration. This way is based on the discovery of the {semigroup property} property, i.e.,~\eqref{tran1}. We show that our accelerated method converge rapidly with the rate of convergence of any desired order. Additionally, this method can be used to solve the unique negative definite solution of~\eqref{eq:NME}, once it exists. The investigation of sufficient conditions for the existence of the negative definite solution of~\eqref{eq:NME} is also included in this work.

\section*{Appendix}

\subsection{Convergence analysis of the fixed point iteration:  $X=\mathcal{F}_\pm(X)$}\label{Sec:fix}
We start our analysis by discussing the convergence property of the DARE. From Corollary~\eqref{cor:main}, we know that the DARE~\eqref{eq:DARE} has a unique positive definite solution $Z_\ast$ if $H_1>0$ and $G_1>0$. Let
\[
Z_{k+1} = H_1+ A_1^HZ_k\Delta_{G_1,Z_k}A_1
\]
be the fixed point iteration of~\eqref{eq:DARE} with an initial positive definite matrix $Z_1$. Like the discussion in Section~\ref{2.1}, we immediately have the following two results:
\begin{itemize}

\item[(a)] The sequence $\{Z_k\}$ is monotone increasing if and only if $Z_1 \leq Z_2$; the sequence $\{Z_k\}$ is monotone decreasing if and only if $Z_1 \geq Z_2$.

\item[(b)] If $Z_\ast \geq Z_1$, then $Z_\ast$ is an upper bounded of the sequence$\{Z_k\}$; if $Z_\ast \leq Z_1$, then $Z_\ast$ is a lower bounded of the sequence $\{Z_k\}$. Moreover, we have $\lim\limits_{k\rightarrow\infty} Z_k = Z_\ast$ in either case.

\end{itemize}

Taking $0<Z_1\leq H_1$, for example, we see that
the sequence $\{Z_k\}$ is monotone increasing, $Z_\ast\geq Z_k$ for all $k$, and $\lim\limits_{k\rightarrow\infty} Z_k = Z_\ast$.
Moreover,
\begin{align}\label{eq:zk}
&Z_\ast- Z_{k+1} \nonumber\\
&= A_1^H\Delta_{{Z}_\ast,G_1}({Z}_\ast(I+G_1 Z_k)-(I+Z_\ast G_1){Z}_k)\Delta_{G_1,{Z}_k}A_1\nonumber\\
&= A_1^H\Delta_{{Z}_\ast,G_1}({Z}_\ast-{Z}_k)\Delta_{G_1,{Z}_k}A_1=T_{Z_\ast}^H(Z_\ast-Z_k)T_{Z_k}\nonumber\\
&=T_{Z_\ast}^H(Z_\ast-Z_k)T_{Z_\ast}+T_{Z_\ast}^H(Z_\ast-Z_k)(T_{Z_k}-T_{Z_\ast})\nonumber\\
&=T_{Z_\ast}^H(Z_\ast-Z_k)T_{Z_\ast}+T_{Z_\ast}^H(Z_\ast-Z_k)\Delta_{G_1,{Z}_k}(I+G_1Z_\ast-I-G_1Z_k)T_{Z_\ast}\nonumber\\
&= T_{Z_\ast}^H({Z}_\ast-{Z}_k)T_{Z_\ast}+T_{Z_\ast}^H\left[ ({Z}_\ast-{Z}_k) G_1\Delta_{{Z}_k,G_1}  ({Z}_\ast-{Z}_k)\right]T_{Z_\ast},
\end{align}
where $T_{Z_\ast} = \Delta_{G_1,Z_\infty} A_1$.
Given a positive number $\epsilon>0$, there exists a positive integer $k_0$
such that
\[
{Z}_\ast-{Z}_k\leq \epsilon I,
\]
for any positive integer $k\geq k_0$.
Since $G\Delta_{{Z}_k,G}\leq G \leq m I$ for a
sufficiently large $m$, it follows from~\eqref{eq:zk} that for this $k_0>0$ and $k\geq k_0$,
\[
{Z}_\ast-{Z}_{k}\leq (1+\epsilon m)T_{Z_\ast}^H({Z}_\ast-{Z}_{k-1})T_{Z_\ast} \leq (1+\epsilon m)^{k-k_0} (T_{Z_\ast}^H)^{k-k_0}({Z}_\ast-{Z}_{k_0}) T_{Z_\ast}^{k-k_0},
\]
or, equivalently,
\begin{equation}\label{eq:zksup}
\limsup\limits_{k\rightarrow\infty}\sqrt[k]{\|{Z}_\ast-{Z}_k\|}\leq (1+\epsilon m) \rho(T_{Z_\ast})^2.
\end{equation}
Since $\epsilon$ is arbitrary,~\eqref{eq:zksup} induces that
\begin{equation}\label{eq:zksup1}
\limsup\limits_{k\rightarrow\infty}\sqrt[k]{\|{Z}_\ast-{Z}_k\|}\leq \rho(T_{Z_\ast})^2.
\end{equation}

When the sequence $\{Z_k\}$ is monotone decreasing and bounded below. A similar argument yields for the estimation \eqref{eq:zksup1}.
Thus, by~\eqref{eq:zksup1}, our discussion to the convergence analysis of the fixed point method $X=\mathcal{F}_\pm(X)$ is divided into two scenarios:
\par\noindent
\begin{enumerate}
\item Consider the fixed-point iteration $X_{k+1}=\mathcal{F}_{+}(X_k)$ with $X_1=H$.
As is discussed in Section~\ref{2.1}, we know that the sequence $\{X_k\}$ is a monotone increasing matrix sequence. In particular, if the solution $X_\ast$ of~\eqref{eq:NMEP} exists, it can be shown that
\begin{align*}
\limsup\limits_{k\rightarrow\infty}\sqrt[k]{\|{X}_\ast-{X}_{2k}\|} \leq \rho({T}_{1})^2,\\
\limsup\limits_{k\rightarrow\infty}\sqrt[k]{\|{X}_\ast-{X}_{2k+1}\|} \leq \rho({T}_{1})^2,
\end{align*}
since $X_1 = H$, $X_2 = H_1$,
 $F^{(2)}_+(X) = H_1+ A_1^H X\Delta_{G_1,X}A_1$,
 and $T_{1} = \Delta_{G_1,X_\infty} A_1$.
Thus, we have
\begin{align*}
\limsup\limits_{k\rightarrow\infty}\sqrt[k]{\| X_k-X_\ast\|}
\leq\rho({T}_{1}).
\end{align*}

\item Consider the fixed-point iteration $X_{k+1}=\mathcal{F}_{-}(X_k)$ with $X_1=H$.
Note that if $X_i\geq X_j$ for any integer $i,j\geq 1$,
then
\begin{eqnarray*}
X_{i+1} - X_{j+1} &=& A^H \overline{X}_j \Delta_{G,\overline{X}_j} A
- A^H \overline{X}_i \Delta_{G,\overline{X}_i} A\\
&=& A^H [(\overline{X}_j^{-1}+G)^{-1}
-(\overline{X}_i^{-1}+G)^{-1}
] A \leq 0,
\end{eqnarray*}
i.e.,
\begin{equation}\label{eq:Yi}
X_{i+1} \leq X_{j+1},
\end{equation}
if $X_i\geq X_j$ for any integer $i,j\geq 1$.
Also, if $H_1$ and $G_1 > 0$, then by Lemma~\ref{Lem2}, we have
\begin{equation}\label{eq:Yi2}
 0<X_2 \leq X_4 \leq \cdots\leq H,
\end{equation}
since $X_{k+2}=\mathcal{F}_{-}^{(2)}(X_k) =
\mathcal{F}_{-}^{(k-1)}(H)
$ for any even number $k>0$.
By~\eqref{eq:Yi} and~\eqref{eq:Yi2}, it can be seen that
\begin{align*}
H= X_1\geq X_3 \geq  X_5 \geq \cdots > 0.
\end{align*}
Here, the first and last inequality follows from the fact that
\begin{align*}
X_3 &= F_{-}(X_2) = H-A^H \overline{X}_2 \Delta_{G,\overline{X}_2 }A \leq H = X_1, \\
X_1 &\geq X_2 > 0,\quad X_3\geq X_4> 0,
\end{align*}
and so on. Upon using the fact that the positive definite solution of
\[
X = F_{-}^{(2)} (X)
\]
is unique once $H_1$ and $G_1 > 0$, we know that $\lim\limits_{k\rightarrow\infty} X_{2k}=\lim\limits_{k\rightarrow\infty} X_{2k+1}:=X_\ast$ , where $X_\ast$ is the unique positive definite solution of~\eqref{eq:NMEM}. Furthermore,
\begin{align*}
\limsup\limits_{k\rightarrow\infty}\sqrt[k]{\|{X}_\ast-{X}_{2k}\|} \leq \rho({T}_{1})^2,
\end{align*}
since $X_2 = H_1$ and $ F^{(2)}_-(X) = H_1+ A_1^H X\Delta_{G_1,X}A_1$. Note that
%
\begin{align*}
 X_{2k+1} -X_\ast &=
 {T}_{1}^H({X}_{2k-1}-{X}_\ast) {T}_{1}\\
 &+ {T}_{1}^H\left[ ({X}_{2k-1}-{X}_\ast) G\Delta_{{X}_{2k-1},G}  ({X}_{2k-1}-{X}_\ast)\right] {T}_{1}.
\end{align*}
for any positive integer $k\geq 1$.
Like the discussion of~\eqref{eq:zksup1}, we have
\begin{align*}
\limsup\limits_{k\rightarrow\infty}\sqrt[k]{\|  {X}_\ast - {X}_{2k+1} \|}\leq \rho( {T}_1)^2.
\end{align*}
This implies that
%
We conclude that
\begin{align*}
\limsup\limits_{k\rightarrow\infty}\sqrt[k]{\| X_k-X_\ast\|} \leq
\rho({T}_{1}).
\end{align*}

\end{enumerate}

\subsection{The proof of Theorem~\ref{trans}}\label{sec:trans}

\begin{proof}
To simply our discussion, let $\Delta_{i,j}:=(I+G_{i}H_{j})^{-1}$
for all $i,j \in\mathbb{N}$. Then, we have
%
\begin{align*}
H_j \Delta_{i,j} &=
H_j^H(I+G_i^HH_j^H)^{-1}
= (I+H_j^HG_i^H)^{-1}H_j^H
=
\Delta_{i,j}^H H_{j}, \\
\Delta_{i,j} G_i &=
(I+G_i^HH_j^H)^{-1}G_i^H
= G_i^H (I+H_j^HG_i^H)^{-1}
=
G_i\Delta_{i,j}^H,\\
I-H_j \Delta_{i,j} G_i &=
I - H_j G_i (I+G_iH_j)^{-1} = (I+H_jG_i)^{-1}
= 
\Delta_{i,j}^H.
\end{align*}

For each $i$, we will prove~\eqref{tran1} by induction with respect to $j$. The proof is divided into two parts. First, for $i=1$, we show that
\begin{align*}
A_{1+j}&=A_{j}(I+G_{1}H_{j})^{-1}A_{1},\\
G_{1+j}&=G_{j}+A_{j}(I+G_{1}H_{j})^{-1}G_{1}A_{j}^H,\\
H_{1+j}&=H_{1}+A_{1}^H H_{j}(I+G_{1}H_{j})^{-1}A_{1}.
\end{align*}
 by induction. Note that for $j=1$, it is trivial from the definition of $A_{2}$, $G_{2}$ and $H_{2}$. Now suppose that it is true for $j=s$.
It follows from Lemma~\ref{Schur} and~\eqref{fix} that
\begin{align}
\Delta_{1,s+1}
&
 =\left(I+G_1\left(H_s+A_s^HH_1(I+G_sH_1)^{-1}A_s\right)\right)^{-1}\nonumber\\
&=\Delta_{1,s}-\Delta_{1,s}(G_1A_{s}^H H_1)\left((I+G_sH_1)+A_s\Delta_{1,s}(G_1A_s^HH_1)\right)^{-1}A_{s}\Delta_{1,s}\nonumber\\
&=\Delta_{1,s}-\Delta_{1,s}G_1A_{s}^H H_1\Delta_{s+1,1}A_s\Delta_{1,s}, \label{31}
\\
\Delta_{s+1,1} &
= \left(I+(G_s+A_s(I+G_1H_s)^{-1}G_1A_s^H)H_1\right)^{-1}
\nonumber\\
&=\Delta_{s,1}-\Delta_{s,1}A_{s}\left((I+G_1H_s)
+ (G_1A_s^HH_1)\Delta_{s,1}A_{s}\right)^{-1}G_1A_s^HH_1 \Delta_{s,1}\nonumber\\
&=\Delta_{s,1}-\Delta_{s,1}A_{s}\Delta_{1,s+1}G_1A_{s}^H H_1\Delta_{s,1}, \label{32}
\\
\Delta_{s+1,1} &=\Delta_{s,1}(I+G_s H_1)\Delta_{s+1,1}
\nonumber\\
&=\Delta_{s,1}(I+(G_{s+1}-A_{s}\Delta_{1,s}G_1A_{s}^H) H_1)\Delta_{s+1,1\nonumber}\\
&=\Delta_{s,1}\left(I+G_{s+1}H_1-A_{s}\Delta_{1,s}G_1A_{s}^H H_1\right)\Delta_{s+1,1}\nonumber\\
&=\Delta_{s,1}-\Delta_{s,1}A_{s}\Delta_{1,s}G_1A_{s}^H H_1\Delta_{s+1,1}. \label{33} 
\end{align}
Then, by induction hypothesis, we have
\begin{align*}
A_{1+(s+1)}
&=A_{1}\Delta_{s+1,1}A_{s+1},\\
&=A_{1}\Delta_{s,1}\left((I
+G_{s+1} H_1)
-
A_s\Delta_{1,s}G_1A_{s}^H H_1\right)  \Delta_{s+1,1}
A_{s+1}
\\
&=A_{1}\Delta_{s,1}\left(I-
A_s\Delta_{1,s}G_1A_{s}^H H_1 \Delta_{s+1,1}\right)
A_{s} \Delta_{1,s}
A_{1}
\\
&=A_{1}  \Delta_{s,1} A_{s}
\left(
\Delta_{1,s}-\Delta_{1,s}G_1A_{s}^H H_1 \Delta_{s+1,1}A_{s} \Delta_{1,s}\right)A_{1}
 \mbox{ (by \eqref{31})}
\\
&=A_{s+1} \Delta_{1,s+1} A_{1},\\
%
%
%
%
%
%
%
G_{1+(s+1)}
&=G_{1}+A_{1}\Delta_{s+1,1}G_{s+1}A_{1}^H,\\
&= G_{1}+A_{1}
(\Delta_{s,1}-\Delta_{s,1}A_s\Delta_{1,s+1}G_1A_s^HH_1\Delta_{s,1})
(G_s+A_s\Delta_{1,s}G_1A_s^H)
A_1^H \mbox{ (by \eqref{32})}
\\
&
=G_{1}+A_{1}\Delta_{s,1}G_{s}A_{1}^H\\
&
-A_{1}\Delta_{s,1}A_{s}\left(\Delta_{1,s+1}G_1A_{s}^H H_1\Delta_{s,1}G_{s}\right)A_{1}^H\\
&
+A_{1}\Delta_{s,1}A_{s}\left(\Delta_{1,s}G_{1}A_{s}^H\right)A_{1}^H\\
& -A_{1}\Delta_{s,1}A_{s}\left(\Delta_{1,s+1}G_1A_{s}^H H_1\Delta_{s,1}A_{s}\Delta_{1,s}G_{1}A_{s}^H\right)A_{1}^H
\\
&=G_{s+1}-A_{s+1}\Delta_{1,s+1}G_1A_{s}^H H_1\Delta_{s,1}G_{s}A_{1}^H\\
&+A_{s+1}\left(I-\Delta_{1,s+1}G_1A_{s}^H H_1\Delta_{s,1}A_{s}\right)\Delta_{1,s}G_{1}A_{s}^HA_{1}^H\\
&=G_{s+1}+A_{s+1}\Delta_{1,s+1}G_1A_{s}^H \left(
I-H_1\Delta_{s,1}G_{s}\right)A_{1}^H\\
&=G_{s+1}+A_{s+1}\Delta_{1,s+1}G_1A_{s}^H
 (
I + H_1 G_{s})^{-1}
A_{1}^H\\
&=G_{s+1}+A_{s+1}\Delta_{1,s+1}G_{1}A_{s+1}^H,
%
%
%
%
\end{align*}
where $I-\Delta_{1,s+1}G_1A_{s}^H H_1\Delta_{s,1}A_{s}= \Delta_{1,s+1}\Delta_{1,s}^{-1}$, and finally,
\begin{align*}
H_{1+(s+1)}
&=
H_{s+1} + A_{s+1}^HH_1\Delta_{s+1,1}A_{s+1},\\
&=H_{s+1}+\left (A_{1}^H\Delta_{1,s}^HA_{s}^H\right ) H_1\Delta_{s+1,1}\left(A_s\Delta_{1,s}A_{1}\right)\\
&=
H_{s+1} + A_1^H(I+H_sG_1)^{-1}A_s^HH_1
\Delta_{s+1,1}A_s\Delta_{1,s}A_{1}
\\
&=H_{s+1}+A_{1}^H\left (I- H_s\Delta_{1,s}G_1\right )A_{s}^H H_1\Delta_{s+1,1}A_s\Delta_{1,s}A_{1}\\
&=H_{1}+A_{1}^H H_{s}\Delta_{1,s}A_{1}\\
&-A_{1}^H H_s\Delta_{1,s}G_1A_{s}^H H_1\Delta_{s+1,1}A_s\Delta_{1,s}A_{1}\\
&+A_{1}^H A_{s}^H H_{1}\Delta_{s,1}\left(I-A_{s}\Delta_{1,s} G_1 A_{s}^H H_1\Delta_{s+1,1}\right)A_{s}\Delta_{1,s}A_{1}\\
&=H_{1}+A_{1}^H H_{s}\Delta_{1,s}A_{1}\\
&-A_{1}^H H_s\Delta_{1,s}G_1A_{s}^H H_1\Delta_{s+1,1}A_s\Delta_{1,s}A_{1}\\
&+A_{1}^H A_{s}^H H_{1}\Delta_{s,1}A_{s}\Delta_{1,s}A_{1}\\
&-A_{1}^H A_{s}^H H_{1}\Delta_{s,1}A_{s}\Delta_{1,s}G_1A_{s}^H H_1\Delta_{s+1,1}A_s\Delta_{1,s}A_{1} \mbox{ (by \eqref{33})}
\\
&= H_{1}+A_{1}^H H_{s+1}\Delta_{1,s+1}A_{1},  \mbox{ (by \eqref{31})}
\end{align*}
where $I-A_{s}\Delta_{1,s} G_1 A_{s}^H H_1\Delta_{s+1,1} = \Delta_{s,1}^{-1}\Delta_{s+1,1}$, which completes the proof for $i=1$.

Assume that~\eqref{tran1} is true for $i = s$ and any integer $j > 0$. Then,  for any integer $j > 0$, we have
\begin{align}
\Delta_{s+1,j}
&= ((I+G_sH_j) + A_s(I+G_1H_s)^{-1}G_1A_s^HH_j)^{-1} \nonumber\\
&=\Delta_{s,j}-\Delta_{s,j}A_{s}
[ (I+G_1H_s)+ G_1A_s^HH_j(I+G_sH_j)A_s
]^{-1}
G_1 A_s^HH_j\Delta_{s,j} \nonumber \\
&=\Delta_{s,j}-\Delta_{s,j}A_{s}\Delta_{1,s+j}G_1 A_s^HH_j\Delta_{s,j}, \label{note1}\\
 \Delta_{1,s+j}
&= ((I+G_1H_s) + G_1A_s^HH_j(I+G_sH_j)^{-1}A_s)^{-1}
\nonumber\\
&=\Delta_{1,s}-\Delta_{1,s}G_1A_{s}^H H_j
[(I+G_s H_j)+A_s \Delta_{1,s}G_1A_s^HH_j]^{-1}
A_s\Delta_{1,s} \nonumber
\\
&=\Delta_{1,s}-\Delta_{1,s}G_1A_{s}^H H_j\Delta_{s+1,j}A_s\Delta_{1,s}, \label{note2}\\
\Delta_{s,j+1}&=
(I+G_s H_{j+1})^{-1} = \left(I+G_s(H_1+A_1^HH_j\Delta_{1,j}A_1)\right)^{-1} \nonumber\\
&=\Delta_{s,1}-\Delta_{s,1}G_{s} A_1^H H_j(I+G_1H_j+A_1\Delta_{s,1}G_sA_1^HH_j)^{-1}A_1\Delta_{s,1} \nonumber \\
&=
\Delta_{s,1}-\Delta_{s,1}G_{s} A_1^H H_j\Delta_{s+1,j}A_1\Delta_{s,1}, \label{note3}
\end{align}
 and
\begin{align}
&
\Delta_{s+1,j}G_{s+1}-\Delta_{s,j}G_{s} \nonumber\\
&=
(\Delta_{s,j}-\Delta_{s,j}A_{s}\Delta_{1,s+j}G_1 A_s^HH_j\Delta_{s,j})
(G_s + A_s \Delta_{1,s}G_1A_s^H)-\Delta_{s,j}G_{s}  \mbox{ (by \eqref{note1})}
\nonumber\\
&=\Delta_{s,j}A_{s}\Delta_{1,s} G_1A_s^H-\Delta_{s,j}A_{s}\Delta_{1,s+j}G_1 A_s^H H_j\Delta_{s,j}G_s-\Delta_{s,j}A_{s}\Delta_{1,s+j}G_1 A_s^H H_j\Delta_{s,j}A_{s}\Delta_{1,s} G_1A_s^H \nonumber
\\
&=\Delta_{s,j}A_{s} \left( I-\Delta_{1,s+j}G_1 A_s^H H_j\Delta_{s,j}A_{s} \right)\Delta_{1,s} G_1A_s^H-\Delta_{s,j}A_{s}\Delta_{1,s+j}G_1 A_s^H H_j\Delta_{s,j}G_s
\nonumber
\\
&=\Delta_{s,j}A_{s}\Delta_{1,s+j}\Delta_{1,s}^{-1} \Delta_{1,s} G_1A_s^H-\Delta_{s,j}A_{s}\Delta_{1,s+j}G_1 A_s^H H_j\Delta_{s,j}G_s
\nonumber
\\
&=\Delta_{s,j}A_{s}\Delta_{1,s+j}G_1 A_s^H\left(I-H_j\Delta_{s,j}G_s\right)
\nonumber
\\
&=\Delta_{s,j}A_{s}\left( \Delta_{1,s+j}G_1\right)(\Delta_{s,j}A_{s})^H, \label{3}\\
&
H_{s+j}\Delta_{1,s+j}-H_{s}\Delta_{1,s} \nonumber\\
&=
(H_s+A_s^HH_j\Delta_{s, j} A_s)(\Delta_{1,s}-\Delta_{1,s}G_1A_{s}^H H_j\Delta_{s+1,j}A_s\Delta_{1,s})
-H_{s}\Delta_{1,s}  \mbox{ (by \eqref{note2})}
\nonumber\\
&=A_{s}^H H_j\Delta_{s,j}A_s\Delta_{1,s}-H_s\Delta_{1,s} G_1A_{s}^H H_j\Delta_{s+1,j}A_s\Delta_{1,s}\nonumber\\
&
-A_{s}^H H_j\Delta_{s,j}A_s\Delta_{1,s}G_1A_{s}^H H_j\Delta_{s+1,j}A_s\Delta_{1,s}\nonumber\\
&=A_{s}^H H_j\Delta_{s,j} \left( I-A_s\Delta_{1,s}G_1A_{s}^H H_j\Delta_{s+1,j} \right)A_s\Delta_{1,s}-H_s\Delta_{1,s} G_1A_{s}^H H_j\Delta_{s+1,j}A_s\Delta_{1,s}\nonumber\\
&=A_{s}^H H_j\Delta_{s,j}\Delta_{s,j}^{-1} \Delta_{s+1,j} A_s\Delta_{1,s}-H_s\Delta_{1,s} G_1A_{s}^H H_j\Delta_{s+1,j}A_s\Delta_{1,s}\nonumber\\
&=\left(I-H_s\Delta_{1,s}G_1\right) A_s^H H_j\Delta_{s+1,j}A_s\Delta_{1,s}\nonumber\\
&=(A_s\Delta_{1,s})^H\left( H_j\Delta_{s+1,j}\right)A_s\Delta_{1,s}. \label{4}
\end{align}
Thus, it follows from Lemma~\ref{Schur} and induction hypothesis that the following result holds for $i =s+1$ and any integer $j>0$. %
 \begin{align*}
A_{(s+1)+j}&=A_{s+(j+1)}=A_{1+j} \Delta_{s,j+1} A_{s}\\
&=A_{j}\Delta_{1,j}A_{1}\left(\Delta_{s,1}-\Delta_{s,1}G_sA_{1}^H H_j\Delta_{s+1,j} A_1\Delta_{s,1}\right)A_{s}
\mbox{ (by \eqref{note3})}
\\
&
=A_{j}\left(\Delta_{1,j}-\Delta_{1,j}A_{1}\Delta_{s,1}G_sA_{1}^H H_j\Delta_{s+1,j}\right)A_{1}\Delta_{s,1}A_{s}\\
&=A_{j}\left(\Delta_{1,j}-\Delta_{1,j}A_{1}\Delta_{s,1}G_sA_{1}^H H_j\Delta_{s+1,j}\right)A_{s+1}\\
&=A_{j}
\Delta_{1,j}
\left( I+G_{s+1}H_j-A_{1}\Delta_{s,1}G_sA_{1}^H H_j\right) \Delta_{s+1,j} A_{s+1}
\\&=A_{j}\Delta_{1,j} \left ( I+(G_{s+1}-A_{1}\Delta_{s,1}G_sA_{1}^H) H_j\right )\Delta_{s+1,j}A_{s+1}
\\
&=A_{j}\Delta_{1,j} \left ( I+G_1H_j\right )\Delta_{s+1,j}A_{s+1}
\\
&=A_{j}\Delta_{s+1,j}A_{s+1},\\
%
G_{s+1+j}&=G_{1+(s+j)}=G_{s+j}+A_{s+j} \Delta_{1,s+j}G_{1} A_{s+j}^H\\
&=(G_j+A_j \Delta_{s,j} G_s A_j^H)+A_j\left((\Delta_{s,j}A_s) (\Delta_{1,s+j}G_{1})(\Delta_{s,j}A_s)^H\right) A_j^H\\
&=G_{j}+A_{j} \Delta_{s+1,j} G_{s+1} A_{j}^H. \mbox{ (by \eqref{3})}\\
%
H_{s+1+j}&=H_{1+(s+j)}=H_1+A_1^H H_{s+j} \Delta_{1,s+j} A_1\\
&=H_1+A_1^H H_{s} \Delta_{1,s} A_1+A_1^H (H_{s+j}\Delta_{1,s+j}-H_{s}\Delta_{1,s} ) A_1\\
&
=
(H_1+A_1^H H_{s} \Delta_{1,s} A_1)+A_1^H \left (
(
A_s\Delta_{1,s})^H ( H_j\Delta_{s+1,j} )A_s\Delta_{1,s}
\right) A_1  \mbox{ (by \eqref{4})}\\
&=H_{s+1}+A_{s+1}^H H_{j} \Delta_{s+1,j} A_{s+1}.
\end{align*}
Now, the induction process is  completed and thus the result is followed.\end{proof}

\subsection{The proof of Lemma~\ref{lem:conv}}\label{sec:3}

\begin{proof}
%
%
Observe that $T_k = T_1^k$ is definitely true for $k=1$. Suppose $T_k$ is true for some $k\geq 1$. Then, by using the fact that
{\[
H_\infty = H_1 + A_1^H H_\infty \Delta_{G_1,H_\infty} A_1,\,
A_{k+1} = A_1 \Delta_{G_k,H_1} A_k,
 \mbox{ and }
G_{k+1} = G_1 + A_1 \Delta_{G_k,H_1} G_k A_1^H,
\]}
we have
\begin{align*}
&T_1^{k+1}=\Delta_{G_1,H_\infty} (A_1\Delta_{G_k,H_\infty} A_k)
\\
&{=\Delta_{G_1,H_\infty} (A_1
(I + G_kH_1+ G_kA_1^H H_\infty \Delta_{G_1,H_\infty} A_1)^{-1}
 A_k)}
\\
&{=\Delta_{G_1,H_\infty}  A_1
[\Delta_{G_k, H_1}
-\Delta_{G_k, H_1} G_kA_1^H H_\infty
\left((I+G_1H_\infty) +
A_1 \Delta_{G_k, H_1} G_kA_1^H H_\infty
\right)^{-1}
A_1\Delta_{G_k, H_1}
]A_k
}
\\
&=\Delta_{G_1,H_\infty} A_{k+1}-\Delta_{G_1,H_\infty} A_1\left(\Delta_{G_k,H_1} G_kA_1^H H_\infty\Delta_{G_{k+1},H_\infty} A_1                       \Delta_{G_k,H_1} \right) A_k \\
&=\Delta_{G_1,H_\infty} \left(I+G_{k+1} H_\infty- A_1\Delta_{G_k,H_1}G_k A_1^H H_\infty \right)\Delta_{G_{k+1},H_\infty} A_{k+1} \\
&=\Delta_{G_{k+1},H_\infty} A_{k+1} = T_{k+1},
\end{align*}
which concludes that $T_k$ holds for all $k\geq 1$.

Observe that $S_k = S_1^k$ is definitely true for $k=1$. Suppose $S_k$ is true for some $k\geq 1$. Then, by using the fact that
{\[
G_\infty = G_1 + A_1G_\infty \Delta_{H_1,G_\infty} A_1^H ,\,
A_{k+1} = A_k \Delta_{G_1,H_k} A_1,
 \mbox{ and }
H_{k+1} = H_1 + A_1^H  \Delta_{H_k,G_1} H_k A_1,
\]}
we have
\begin{align*}
&S_1^{k+1}
=
 (A_k\Delta_{G_\infty,H_k})A_1
\Delta_{G_\infty,H_1}
\\
&{=
 A_k [I+(G_1
+A_1 G_\infty \Delta_{H_1,G_\infty}A_1^H))H_k]^{-1}
A_1\Delta_{G_\infty,H_1}
}
\\
&{=
 A_k [\Delta_{G_1,H_k}
 - \Delta_{G_1,H_k} A_1 G_\infty
 (I+H_1 G_\infty + A_1^H H_k \Delta_{G_1,H_k}A_1G_\infty)^{-1}A_1^HH_k
  \Delta_{G_1,H_k}
 ]
A_1\Delta_{G_\infty,H_1}
}
\\
&=A_{k+1} \Delta_{G_\infty,H_1} -
A_{k+1}G_\infty(I+H_{k+1}G_\infty)^{-1} A_1^HH_k
  \Delta_{G_1,H_k}
A_1\Delta_{G_\infty,H_1}
\\
&=A_{k+1} \Delta_{G_\infty,H_{k+1}}
(I + G_\infty H_{k+1}
-G_\infty
A_1^HH_k
  \Delta_{G_1,H_k}
A_1
 )
\Delta_{G_\infty,H_1} \\
&= A_{k+1}
\Delta_{G_\infty,H_{k+1}} = S_{k+1} ,
\end{align*}
which concludes that $S_k$ holds for all $k\geq 1$. 
Note that
\begin{align*}
&H_\infty-H_{k+1}\\
&=A_1^H\left(H_\infty\Delta_{G_1,H_\infty} -H_k\Delta_{G_1,H_k}\right)A_1\\
&{ = A_1^H H_\infty\Delta_{G_1,H_\infty}(I+G_1H_k)\Delta_{G_1,H_k}A_1
- A_1^H \Delta_{H_\infty, G_1}(I+H_\infty  G_1)H_k \Delta_{G_1,H_k}A_1}\\
%
%
&=A_1^H\Delta_{H_\infty,G_1}\left(H_\infty -H_k\right)\Delta_{G_1,H_k}A_1\\
&{ =A_1^H\Delta_{H_\infty,G_1}
(
A_k^H H_\infty \Delta_{G_k, H_\infty} A_k
)
\Delta_{G_1,H_k}A_1}\\
%
%
%
&{=T_1^H T_k^H H_\infty A_{k+1}}
\\
&=T_{k+1}^H H_\infty A_{k+1} =T_{k+1}^H( H_\infty^{-1} +G_{k+1}) T_{k+1}.
\end{align*}
{Also,
\begin{align*}
&G_\infty-G_{k+1}\\
&=A_1\left(\Delta_{G_\infty,H_1}
G_\infty
 -\Delta_{G_k,H_1} G_k\right)A_1^H\\
& = A_1(G_\infty\Delta_{H_1,G_\infty}
(I+H_1G_k)\Delta_{H_1,G_k}A_1^H
- A_1 \Delta_{G_\infty, H_1}(I+G_\infty  H_1)G_k \Delta_{H_1,G_k}A_1^H\\
%
%
&=A_1\Delta_{G_\infty,H_1}\left(G_\infty -G_k\right)\Delta_{H_1,G_k}A_1^H\\
& =A_1\Delta_{G_\infty,H_1}
(
A_k  \Delta_{G_\infty, H_k}G_\infty A_k^H
)
\Delta_{H_1,G_k}A_1^H\\
&={S}_1 {S}_k G_\infty A_{k+1}^H
\\
&={S}_{k+1} G_\infty A_{k+1}^H ={S}_{k+1} G_\infty (I+H_{k+1} G_\infty) {S}_{k+1}^H.
\end{align*}
}

By the similar discussion of Theorem~\ref{1ALremark} and Theorem~\ref{thm23} for Eq.~\eqref{eq:NME}, we can show with no difficulty that $\rho(\Delta_{H_1,G_\infty}A_1^H)<1$
with respect to Eq.~\eqref{du}, that is,
\[\rho(S_1)=\rho((\Delta_{H_1,G_\infty}A_1^H)^H)<1.
\]

On the other hand, let  $\mathcal{M}$ and $\mathcal{L}$ be two matrices defined by
\begin{equation*}
   \mathcal{M} =
   \left [
   \begin{array}{rc} A_{1} & 0 \\ -H_{1} & I_n
   \end{array} \right ] , \quad
   \mathcal{L}:=\left [ \begin{array}{lc}I_n & G_{1} \\ 0 & A_{1}^H \end{array}
   \right ].
\end{equation*}
Let $\mathcal{J}$ be a skew-symmetric matrix defined by
\[
\mathcal{J} \left[\begin{array}{cc}0 &  I_n \\ -I_n & 0\end{array}\right]
\]
It can be seen that $\mathcal{M}\mathcal{J}\mathcal{M}^H=\mathcal{L}\mathcal{J}\mathcal{L}^H$, since $G_1 = G_1^H$ and $H_1 = H_1^H$. Let $\lambda\in\sigma(\mathcal{M}-\lambda \mathcal{L})$ and $\mathbf{x}$ be a nonzero eigenvector satisfying $\mathcal{M}^H \mathbf{x} = \overline{\lambda} \mathcal{L}^H\mathbf{x}$.
It follows that
\[
\mathcal{L}\mathcal{J}\mathcal{L}^H \mathbf{x} =\mathcal{M}\mathcal{J}\mathcal{M}^H \mathbf{x}=\overline{\lambda}\mathcal{M}\mathcal{J}\mathcal{L}^H x ,
\]
First, if $\lambda \neq 0$, we have
$\mathcal{M} (\mathcal{J}\mathcal{L}^H \mathbf{x})=({1}/{\overline{\lambda}} )\mathcal{L} (\mathcal{J}\mathcal{L}^H \mathbf{x})$. Once $\mathcal{J}\mathcal{L}^H \mathbf{x}\neq 0$, this implies that
$1/\overline{\lambda}\in \sigma(\mathcal{M}-\lambda \mathcal{L})$; otherwise,
$\mathcal{M}^H \mathbf{x} = 0$ if $\mathcal{J}\mathcal{L}^H \mathbf{x}= 0$,
which contradicts that $\mathbf{x}$ is nonzero.
%
%
%
%
%
%
%
Second, if 
$\lambda = 0$, there exists
a nonzero vector $\mathbf{x}$ such that
 $\mathcal{M}^H \mathbf{x} = 0$.
 {Since $\text{rank}(\mathcal{M})=\text{rank}(\mathcal{L})$, it follows that there exists a nonzero vector  $\mathbf{y}$
 such that $\mathcal{L} \mathbf{y}= 0$
 and, hence, $\infty:=1/0\in   \sigma(\mathcal{M}-\lambda \mathcal{L})$.}
Thus, the eigenvalues of $\mathcal{M}-\lambda \mathcal{L}$ come in pairs, i.e., $
1/\overline{\lambda} \in \sigma(\mathcal{M}-\lambda \mathcal{L})$ if $
\lambda \in \sigma(\mathcal{M}-\lambda \mathcal{L})$.

Let $U=\bb I_n \\ H_\infty\eb$ and $V=\bb -G_\infty \\ I_n\eb$. It is true that
\begin{align*}
\mathcal{M} U=\mathcal{L}  U T_1, \quad \mathcal{M} V {S_1}^H =\mathcal{L}  V.
\end{align*}
This implies that $\sigma(T_1)\subset\sigma(\mathcal{M}-\lambda \mathcal{L})$ and $\sigma(S_1)\subset\sigma(\mathcal{L}-\lambda \mathcal{M})$.  {Furthermore, there are exactly $n$ eigenvalues of $\mathcal{M}-\lambda \mathcal{L}$ inside the unit circle and the other outside the unit circle, since $\rho({T_1}) < 1$.}

If $\lambda\in\sigma(T_1)$ ( i.e.,
$1/\overline{\lambda}\in\sigma(\mathcal{M}-\lambda \mathcal{L})$), then there exists a $x\neq 0$ such that $\overline{\lambda}\mathcal{M} x=\mathcal{L}x$ and, hence, $\overline{\lambda}\in\sigma(S_1)$. The converse is also true and concludes that $\sigma(T_1)=\sigma(S_1^H)$.
\end{proof}

\section*{Acknowledgment}
This research work is partially supported by the Ministry of Science and Technology and the National Center for Theoretical Sciences in Taiwan. The first author (Matthew M. Lin) like to thank the support from the Ministry of Science and Technology  of Taiwan under
under grants MOST 104-2115-M-006-017-MY3 and 105-2634-E-002-001, and the corresponding author (Chun-Yueh Chiang) like to thank the support from the Ministry of Science and Technology of Taiwan under the grant MOST 105-2115-M-150-001.

%
%


\begin{thebibliography}{10}

\bibitem{Bernstein2009}
D.~S. Bernstein.
\newblock {\em Matrix mathematics}.
\newblock Princeton University Press, Princeton, NJ, second edition, 2009.
\newblock Theory, facts, and formulas.

\bibitem{Chiang2010}
C.-Y. Chiang, H.-Y. Fan, and W.-W. Lin.
\newblock A structured doubling algorithm for discrete-time algebraic {R}iccati
  equations with singular control weighting matrices.
\newblock {\em Taiwanese J. Math.}, 14(3A):933--954, 2010.

\bibitem{Davies2008}
R.~Davies, P.~Shi, and R.~Wiltshire.
\newblock New upper solution bounds of the discrete algebraic {R}iccati matrix
  equation.
\newblock {\em J. Comput. Appl. Math.}, 213(2):307--315, 2008.

\bibitem{Sayed01}
S.~M. El-Sayed and A.~C.~M. Ran.
\newblock On an iteration method for solving a class of nonlinear matrix
  equations.
\newblock {\em SIAM J. Matrix Anal. Appl.}, 23(3):632--645, 2002.

\bibitem{Gudmundsson1992}
T.~Gudmundsson, C.~Kenney, and A.~J. Laub.
\newblock Scaling of the discrete-time algebraic {R}iccati equation to enhance
  stability of the {S}chur solution method.
\newblock {\em IEEE Trans. Automat. Control}, 37(4):513--518, 1992.

\bibitem{Guo1999}
C.-H. Guo.
\newblock Newton's method for discrete algebraic {R}iccati equations when the
  closed-loop matrix has eigenvalues on the unit circle.
\newblock {\em SIAM J. Matrix Anal. Appl.}, 20(2):279--294, 1999.

\bibitem{Kelly1995}
C.~T. Kelley.
\newblock {\em Iterative {M}ethods for {L}inear and {N}onlinear {E}quations},
  volume~16 of {\em Frontiers in Applied Mathematics}.
\newblock Society for Industrial and Applied Mathematics (SIAM), Philadelphia,
  PA, 1995.

\bibitem{Kimura1988}
M.~Kimura.
\newblock Convergence of the doubling algorithm for the discrete-time algebraic
  {R}iccati equation.
\newblock {\em Internat. J. Systems Sci.}, 19(5):701--711, 1988.

\bibitem{WWL2006}
W.-W. Lin and S.-F. Xu.
\newblock Convergence analysis of structure-preserving doubling algorithms for
  {R}iccati-type matrix equations.
\newblock {\em SIAM J. Matrix Anal. Appl.}, 28(1):26--39, 2006.

\bibitem{Lu1993}
L.~Z. Lu and W.~W. Lin.
\newblock An iterative algorithm for the solution of the discrete-time
  algebraic {R}iccati equation.
\newblock {\em Linear Algebra Appl.}, 188/189:465--488, 1993.

\bibitem{Lu1999}
L.~Z. Lu, W.~W. Lin, and C.~E.~M. Pearce.
\newblock An efficient algorithm for the discrete-time algebraic {R}iccati
  equation.
\newblock {\em IEEE Trans. Automat. Control}, 44(6):1216--1220, 1999.

\bibitem{Miyajima2017}
S.~Miyajima.
\newblock Fast verified computation for stabilizing solutions of discrete-time
  algebraic {R}iccati equations.
\newblock {\em J. Comput. Appl. Math.}, 319:352--364, 2017.

\bibitem{Pappas1980}
T.~Pappas, A.~J. Laub, and N.~R. Sandell, Jr.
\newblock On the numerical solution of the discrete-time algebraic {R}iccati
  equation.
\newblock {\em IEEE Trans. Automat. Control}, 25(4):631--641, 1980.

\bibitem{Reurings2003}
M.~Reurings.
\newblock Symmetric matrix equations.
\newblock PhD Thesis, Vrije Universiteit, Amsterdan, 2003, ISBN 90-9016681-5.

\bibitem{Rojas2013}
A.~J. Rojas.
\newblock Explicit solution for a class of discrete-time algebraic {R}iccati
  equations.
\newblock {\em Asian J. Control}, 15(1):132--141, 2013.

\bibitem{Sun1998}
J.-g. Sun.
\newblock Sensitivity analysis of the discrete-time algebraic {R}iccati
  equation.
\newblock In {\em Proceedings of the {S}ixth {C}onference of the
  {I}nternational {L}inear {A}lgebra {S}ociety ({C}hemnitz, 1996)}, volume
  275/276, pages 595--615, 1998.

\bibitem{Zhang2015}
J.~Zhang and J.~Liu.
\newblock New upper and lower bounds, the iteration algorithm for the solution
  of the discrete algebraic {R}iccati equation.
\newblock {\em Adv. Difference Equ.}, pages 2015:313, 17, 2015.

\bibitem{Zhou2011}
B.~Zhou, J.~Lam, and G.-R. Duan.
\newblock Toward solution of matrix equation {$X=Af(X)B+C$}.
\newblock {\em Linear Algebra Appl.}, 435(6):1370--1398, 2011.

\end{thebibliography}
\def\cprime{$'$}

\end{document}